\RequirePackage{fix-cm}

\documentclass[smallextended,final]{svjour3}

\usepackage{amssymb}
\usepackage{stmaryrd}
		\SetSymbolFont{stmry}{bold}{U}{stmry}{m}{n}
\usepackage[right=3.45cm,left=3.45cm,top=4.99cm,bottom=4.99cm,bindingoffset=1cm]{geometry}
\usepackage[automark]{scrpage2}

\usepackage[utf8]{inputenc}
\usepackage[english]{babel}
\usepackage{csquotes}
\usepackage{mathptmx}
\usepackage[T1]{fontenc}
\usepackage[final]{microtype}
\usepackage[mathscr]{euscript} 

\usepackage{amsmath}
\usepackage{mathtools}
\usepackage{stmaryrd}
\SetSymbolFont{stmry}{bold}{U}{stmry}{m}{n}

\usepackage{pgf}
\usepackage{tikz}
\usepackage{tikz-cd}
\usetikzlibrary{arrows, matrix}

\newtheorem{assumption}[definition]{Assumption}

\newcommand{\bb}{\pmb}

\newcommand{\N}{{\mathbb N}}
\newcommand{\R}{{\mathbb R}}

\newcommand{\K}{{\mathbb K}}
\newcommand{\norm}[1]{{\left\lVert #1 \right\rVert}}
\newcommand{\abs}[1]{{\left\lvert #1 \right\rvert}}

\newcommand{\opd}{\:\operatorname{d}}
\renewcommand{\div}{\operatorname{div}}
\DeclareMathOperator{\curl}{{curl}}

\renewcommand{\S}{{\mathbb S}}

\newcommand{\vcurl}{{\bb\curl}}
\DeclareMathOperator{\grad}{grad}

\journalname{arXiv}
\usepackage{multicol}
\usepackage{placeins}
\usepackage{pdfcomment}
 \usepackage{ltablex}
\usepackage{tabularx}
\usepackage{subcaption}
\usepackage{tikz}
\usepackage{pgfplots}
\usepackage{xcolor}

\usetikzlibrary{calc}
\newcommand{\appendA}{{Appendix A}{}}

\title{Multipatch Approximation of the de Rham Sequence and its Traces in Isogeometric Analysis}

\author{Annalisa Buffa${}^2$ \and J\"urgen D\"olz${}^{3}$ \and   Stefan Kurz${}^{3}$ \and \mbox{Sebastian Sch\"ops${}^{3}$} \and  Rafael~Vázquez${}^2$ \and \mbox{Felix Wolf${}^{1,3}$}
}

\institute{
${}^1$ F.~Wolf. Corresponding Author.\at 
Technische Universität Darmstadt, Institute for Accelerator Science and Electromagnetic Fields \& Centre for Computational Engineering, Darmstadt, Germany\\\email{wolf@gsc.tu-darmstadt.de}\\
\emph{Present address:} Dolivostr. 15, 64293 Darmstadt, Germany. Tel.:  
+49\,6151\,16-24401
\and
${}^2$ A.~Buffa, R.~Vázquez\at
École polytechnique fédérale de Lausanne, Chair of Numerical Modelling and Simulation, Lausanne, Switzerland
\email{[annalisa.buffa\textbackslash{}rafael.vazquez]@epfl.ch}
\and
${}^3$ J.~Dölz, S.~Kurz, S.~Schöps\at Technische Universität Darmstadt, Institute for Accelerator Science and Electromagnetic Fields \& Centre for Computational Engineering, Darmstadt, Germany\\{\email{{[doelz\textbackslash{}schoeps\textbackslash{}kurz]@gsc.tu-darmstadt.de}}}
              }

\date{Received: \today / Accepted: --}

\begin{document}

\maketitle
\begin{abstract}
We define a conforming B-spline discretisation of the de Rham complex on multipatch geometries. We introduce and analyse the properties of interpolation operators onto these spaces which commute w.r.t.~the surface differential operators.
Using these results as a basis, we derive new convergence results of optimal order w.r.t.~the respective energy spaces and provide approximation properties of the spline discretisations of trace spaces for application in the theory of isogeometric boundary element methods. Our analysis allows for a straight forward generalisation to finite element methods.
\end{abstract}
\begin{acknowledgements}
The authors want to express their sincere gratitude to Jacopo Corno, who provided feedback on early versions of the draft.
The work of A.~Buffa and R.~V\'azquez was partially supported by the European Research Council through the H2020 ERC Advanced Grant no.~694515 CHANGE.
J.~D\"olz is an \emph{Early Postdoc.Mobility} fellow, funded by the Swiss National Science Foundation through the project 174987 
\emph{H-Matrix Techniques and Uncertainty Quantification in Electromagnetism}, the Excellence Initiative of the German Federal and State Governments and the Graduate School of Computational Engineering at Technische Universität Darmstadt. 
The work of F.~Wolf is supported by DFG Grants SCHO1562/3-1 and KU1553/4-1 within the project 
\emph{Simulation of superconducting cavities with isogeometric boundary elements (IGA-BEM)},
the Excellence Initiative of the German Federal and State Governments and the Graduate School of Computational Engineering at Technische Universität Darmstadt.
\end{acknowledgements}

\section{Introduction}
 
Since its introduction by Hughes \emph{et al.} in \cite{Hughes_2005aa}, the technique of \emph{isogeometric analysis} has sparked interest in various communities, see e.g.~\cite{Bontinck_2017ag,Cottrell_2009aa}.
Modern design tools often represent the geometries via NURBS mappings \cite{Piegl_1997aa}, which, in the framework of isogeometric analysis, are utilised as mappings from reference elements onto an exact representation of the geometry.  
This enables the user to perform simulations without the introduction of geometric errors. 
As discrete function spaces, spaces underlying the parametrisation of the geometry are used; such that forces obtained as the results of numerical simulations can be applied to the geometry in the form of deformations.
This, in theory, unites the design and simulation processes, since the geometry format for simulation and design coincide, thus, eliminating the need for frequent remeshing and preprocessing of the computational domain.
However, in many applications, the geometries are merely given via a boundary representation, i.e., as two-dimensional surfaces in a three-dimensional ambient space. Thus, for many numerical applications that want to utilise the high orders of convergence and spectral properties of isogeometric analysis, a volumetric parametrisation of the computational domain has to be constructed by hand.

For some problems, this issue can be overcome by the use of \emph{boundary element methods}. 
Indeed, many applications of isogeometric boundary element methods have been introduced in recent years \cite{Beer_2017aa,Doelz_2017aa,Doelz_2016aa,Marussig_2015aa,Simpson_2012aa,Simpson_2017aa}. 
These go beyond the scope of academic examples and show that isogeometric boundary methods are ready for industrial application. 
This can be attributed to the application of so-called \emph{fast methods} \cite{Doelz_2016aa,Kurz_2007aa,Hackbusch_2002aa}, which counteract the dense matrices arising from boundary element formulations.
The analysis of classical boundary element methods is well understood, see \cite{McLean_2000aa,Sauter_2010aa} for the scalar cases, and \cite{Buffa_2001aa,Buffa_2001ac,Buffa_2002aa,Buffa_2003aa} for the case of electromagnetic problems, and properties of different choices of discretisation are detailed by \cite{Weggler_2011aa,Zaglmayr_2006aa}, going back to the works of {\cite{Bossavit_1988aa,Bossavit_1998aa,Ciarlet_2002aa,Monk_1993aa,Monk_2003aa,Nedelec_1980aa}} and many more.
Moreover, the utilisation of parametric mappings in the context of boundary element methods is not new. For different choices of basis functions, much of the theory has already been investigated, cf.~\cite{Harbrecht._2001aa,Harbrecht_2013aa}.
However, this kind of analysis has not yet been done for B-splines as ansatz functions and for a full discretisation of the de Rham diagram, as needed for problems requiring divergence conforming discretisations.
With isogeometric boundary element methods in mind, one cannot simply rely on the established analysis of variational isogeometric methods \cite{Veiga_2014aa}. 
Despite the fact, that first multipatch estimates have been investigated in \cite{Buffa_2015aa}, the \emph{spline complex} \cite{Buffa_2010aa}, i.e.~a conforming B-spline discretisation of the de Rham complex, has not been analysed for the multipatch setting. 
Moreover, error analysis in the trace space, i.e., the spaces on the boundary of a domain on which boundary element methods operate, cannot be trivially deduced by an error analysis of finite element methods, since the norms induced on the boundary are nonlocal norms, defined through dualities \cite{McLean_2000aa}.

In this paper, we want to establish approximation estimates of optimal order for the trace spaces $H^{1/2}(\Gamma)$, $\bb H_\times^{-1/2}(\div_\Gamma,\Gamma)$ and $H^{-1/2}(\Gamma),$ where $\Gamma=\partial\Omega$. These spaces and some required definitions will be introduced in Section \ref{sec::definitions}. 
We will use spline-techniques as in \cite{Buffa_2010aa}, going back to \cite{Schumaker_2007aa}, to first define a multipatch spline complex (Section \ref{subsec::complex}). Then, in Section \ref{subsec::approx}, investigate its approximation properties w.r.t.~standard norms on multipatch boundaries. 

In Section \ref{sec::tracespaces}, we will follow the lines of established boundary element literature, e.g.~\cite{Buffa_2003ab,Buffa_2003aa,Sauter_2010aa,Steinbach_2008aa}, and show that isogeometric approximation on trace spaces share the approximation properties of classical alternatives \cite{Weggler_2011aa,Zaglmayr_2006aa}.
Finally, in Section \ref{sec::conclusion}, we will collect the results.

\section{Trace Spaces for Boundary Element Methods} \label{sec::definitions}

We will introduce necessary definitions, and discuss notation. For an in-depth introduction, we refer to the books by Adams \cite{Adams_1978aa} and McLean \cite{McLean_2000aa}. 
Let $\Omega\subseteq \R^3$ be some Lipschitz domain and let $Df$ denote the weak derivative of some function $f$. As in \cite{Buffa_2003aa} or \cite{Nezza_2012aa}, we will follow convention and set $H^0(\Omega)= L^2(\Omega)$. 

For any integer $m\geq 1$, we define $H^m(\Omega)=\lbrace f\in L^2(\Omega)\colon D f \in H^{m-1}(\Omega)\rbrace$ equipped with the norm recursively defined by
\begin{align*}
        \norm{f}_{H^0(\Omega)} & \coloneqq\norm{f}_{L^2(\Omega)},                                                            &
        \norm{f}_{H^m(\Omega)}^2 & \coloneqq\norm{f}_{H^{m-1}(\Omega)}^2 + \sum_{\abs{\bb\alpha} = m}\norm{D^{\bb\alpha} f}_{L^2(\Omega)}^2,
\end{align*}
where ${\bb\alpha}$ is a multiindex with $\abs{ {\bb\alpha}} = \sum_{1\leq i\leq 3} \alpha_i =m$
and $D^{\bb\alpha} f = \big( {\partial^{\alpha_1}_{x_1}},\dots,{\partial^{\alpha_{3}}_{x_{3}}} \big).$
For the special case $H^1(\Omega)$ we find $\norm{f}_{H^1(\Omega)}^2 = \norm{f}_{L^2(\Omega)}^2 + \norm{\bb\grad f}_{\bb L^2(\Omega)}^2$. 
By $\abs{\cdot}_{H^m(\Omega)}$ we will denote the $m$-th semi-norm, i.e., the term with $\norm{\cdot}_{H^m(\Omega)}^2 = \norm{\cdot}_{H^{m-1}(\Omega)}^2 + \abs{\cdot}_{H^m(\Omega)}^2$.

Now let $s = m+\epsilon$, where $m\in \N$ and $\epsilon \in (0,1).$
We define the fractional Sobolev space $H^s(\Omega)$ as the functions of $L^2(\Omega)$ for which the norm 
\begin{align*}
    \norm{f}_{H^s(\Omega)}^2 \coloneqq \norm{f}_{H^m(\Omega)}^2 + {\sum_{\abs{\alpha} = m}\int_\Omega\int_\Omega \frac{\abs{D^\alpha f(\bb x)-D^\alpha f(\bb y)}^2}{\abs{\bb x-\bb y}^{2\epsilon+3}}\opd x\opd y}
\end{align*}
is finite. We equip $H^s(\Omega)$ with the corresponding norm.

Vectorial Sobolev spaces can be defined largely analogously and will be denoted by bold letters, for example $\bb H^s(\Omega)$.

For any first-order differential operator $\operatorname d$, we set $$H^s(\operatorname{d},\Omega)\coloneqq \lbrace f \in H^s(\Omega)  \colon \operatorname d f \in H^s(\Omega)\rbrace,$$ equipped with the corresponding graph norm. Of specific interest are spaces of types \begin{align*}
\bb H^s(\div,\Omega)&\coloneqq \lbrace\bb f\in\bb H^s(\Omega)\colon \div(\bb f)\in H^s(\Omega) \rbrace,\\
\bb H^s(\bb\curl,\Omega)&\coloneqq \lbrace\bb f\in\bb H^s(\Omega)\colon \bb\curl(\bb f)\in \bb H^s(\Omega) \rbrace,
\end{align*}
and spaces of similar structure w.r.t.~the surface differential operators $\bb\grad_\Gamma,$ $\div_\Gamma,$ $\vcurl_\Gamma$ and $\curl_\Gamma$, cf.~\cite{Buffa_2003aa,Peterson_1995aa}.
\subsection{Trace Space Setting}

We are interested in function spaces on compact boundaries of Lipschitz domains $\Gamma = \partial \Omega$. 
As commonly done, we can now define the corresponding spaces on manifolds $\Gamma$ via charts and partitions of unity, cf.~\cite{McLean_2000aa}. 
\begin{definition}[Trace Operators, \cite{Buffa_2003aa,Sauter_2010aa}]
    Let $u\colon\Omega\to \mathbb C$ and $\bb u \colon \Omega \to \mathbb C^3$. Following the notation of \cite{Buffa_2003aa}, we define the \emph{trace operators} for smooth $u$ and $\bb u$ as
    \begin{align*}
            \gamma_0 (u)(\bb x_0)         & \coloneqq \lim_{\bb x\to \bb x_0} u(\bb x),                                                 &
            \bb \gamma_0 (\bb u)(\bb x_0) & \coloneqq \lim_{\bb x\to \bb x_0}\bb u(\bb x) - \bb{ n}_{\bb x_0} (\bb u(\bb x)\cdot \bb{n}_{\bb x_0}),  &\\
            \bb \gamma_{ t} (\bb u)(\bb x_0) & \coloneqq \lim_{\bb x\to \bb x_0}\bb u(\bb x) \times \bb{n}_{\bb x_0},                          &
            \gamma_{\bb n} (\bb u)(\bb x_0)     & \coloneqq \lim_{\bb x\to \bb x_0}\bb u(\bb x) \cdot \bb{n}_{\bb x_0},
    \end{align*}
    for $\bb x_0\in\Gamma$ and $\bb x\in \Omega$, where $\bb{n}_{\bb x_0}$ denotes the {outward} normal vector of $\Omega$ at $\bb x_0\in \Gamma$.
\end{definition}
By density arguments, one extends these operators to a weak setting, see \cite{McLean_2000aa}. 
One can visualise the trace operators acting on vector fields as in Figure \ref{Fig::traces}.
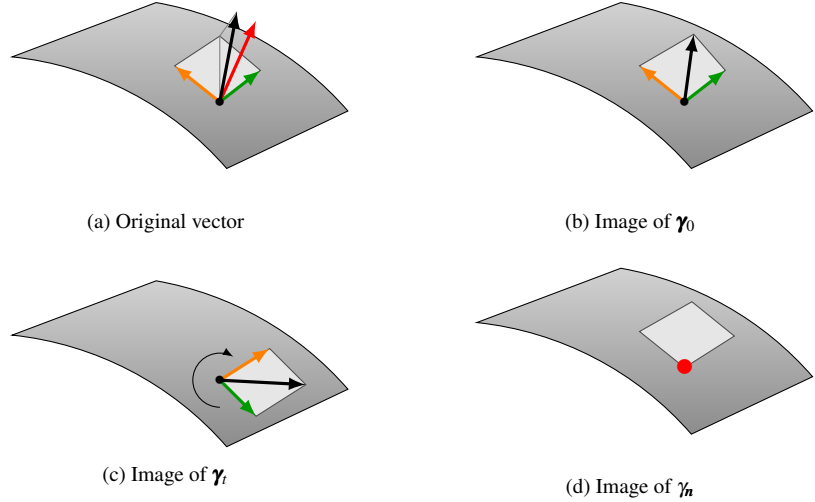
\begin{figure}\centering
	\begin{subfigure}{.3\textwidth }
		\begin{tikzpicture}[scale=1.2] 

\coordinate (Origin) at (2,.5);
\coordinate (T2) at ($(.45,.35)+(Origin)$);
\coordinate (T1) at ($(-.5,.4)+(Origin)$);
\coordinate (Normal) at ($(.4,.9)+(Origin)$);
\coordinate (real) at ($(.19,1)+(Origin)$);
\coordinate (Ghost) at ($(.0,.73)+(Origin)$);
\path [draw=black, top color=gray!35,middle color=gray!60,bottom color=gray!90](-.3,1) arc (264.036:220.964:-3.65) -- (3.42,.4) arc (220.964:260.036:-3.65) -- cycle;
\filldraw[draw=black!70, fill=gray!20] (Origin) -- (T1) -- (Ghost) -- (T2) -- cycle;

\draw [-latex,red,very thick] (Origin) -- (Normal);
\draw [-latex,orange,very thick] (Origin) -- (T1);
\draw [-latex,black!40!green!100,very thick] (Origin) -- (T2);
\draw [gray] (Origin) -- (Ghost);
\draw [gray] (real) -- (Ghost);
\draw [-latex,black,very thick] (Origin) -- (real);
\node at (Origin) {$\bullet$};

\end{tikzpicture}
		\caption{Original vector}
	\end{subfigure}\hspace{2cm}
	\begin{subfigure}{.3\textwidth }
		\begin{tikzpicture}[scale=1.2] 
\coordinate (Origin) at (2,.5);
\coordinate (T2) at ($(.45,.35)+(Origin)$);
\coordinate (T1) at ($(-.5,.4)+(Origin)$);
\coordinate (Normal) at ($(.4,.9)+(Origin)$);
\coordinate (Ghost) at ($(.08,.73)+(Origin)$);
\coordinate (real) at ($(.19,1)+(Origin)$);

\coordinate (Shift) at (5,1);
\coordinate (Origin2) at ($(Origin)+(Shift)$);
\coordinate (T22) at ($(.45,.35)+(Origin2)$);
\coordinate (T12) at ($(-.5,.4)+(Origin2)$);
\coordinate (Normal2) at ($(.4,.9)+(Origin2)$);
\coordinate (Ghost2) at ($(.1,.75)+(Origin2)$);
\coordinate (T-12) at ($(.4,-.4)+(Origin2)$);
\path [draw=black, top color=gray!35,middle color=gray!60,bottom color=gray!90]($(-.3,1)+(Shift)$) arc (264.036:220.964:-3.65) -- ($(3.42,.4)+(Shift)$) arc (220.964:260.036:-3.65) -- cycle;
\filldraw[draw=black!70, fill=gray!20] (Origin2) -- (T12) -- (Ghost2) -- (T22) -- cycle;
\draw [-latex,orange,very thick] (Origin2) -- (T12);
\draw [-latex,black!40!green!100,very thick] (Origin2) -- (T22);
\draw [-latex,black,very thick] (Origin2) --  (Ghost2);

\node at (Origin2) {$\bullet$};
\end{tikzpicture}
		\caption{Image of $\bb\gamma_0$}
	\end{subfigure}\\[.4cm]
	\begin{subfigure}{.3\textwidth }
		\begin{tikzpicture}[scale=1.2] 
\coordinate (Origin) at (2,.5);
\coordinate (T2) at ($(.45,.35)+(Origin)$);
\coordinate (T1) at ($(-.5,.4)+(Origin)$);
\coordinate (Normal) at ($(.4,.9)+(Origin)$);
\coordinate (real) at ($(.19,1)+(Origin)$);
\coordinate (Ghost) at ($(.0,.73)+(Origin)$);
\coordinate (Shift) at (5,1);
\coordinate (Origin2) at ($(Origin)+(Shift)$);

 \draw[-latex] ($(Origin2)-(-.05,.2)$) arc (-60:-290:.22cm);

\coordinate (T22) at ($(.55,.35)+(Origin2)$);
\coordinate (T2) at ($(.45,.35)+(Origin2)$);
\coordinate (Normal2) at ($(.4,.9)+(Origin2)$);
\coordinate (Ghost2) at ($(.95,-.05)+(Origin2)$);
\coordinate (T-12) at ($(.4,-.4)+(Origin2)$);

\path [draw=black, top color=gray!35,middle color=gray!60,bottom color=gray!90]($(-.3,1)+(Shift)$) arc (264.036:220.964:-3.65) -- ($(3.42,.4)+(Shift)$) arc (220.964:260.036:-3.65) -- cycle;
\filldraw[draw=black!70, fill=gray!20] (Origin2) -- (T-12) -- (Ghost2) -- (T22) -- cycle;

\draw[-latex] ($(Origin2)-(0,.3)$) arc (90:-120:-.3);

\draw [-latex,black!40!green!100,very thick] (Origin2) -- (T-12);
\draw [-latex,orange,very thick] (Origin2) -- (T22);
\draw [-latex,black,very thick] (Origin2) -- (Ghost2);
\node at (Origin2) {$\bullet$};
\end{tikzpicture}
		\caption{Image of $\bb \gamma_t$}
	\end{subfigure}\hspace{2cm}
	\begin{subfigure}{.3\textwidth }
		\begin{tikzpicture}[scale=1.2] 
\coordinate (Origin) at (2,.5);
\coordinate (T2) at ($(.45,.35)+(Origin)$);
\coordinate (T1) at ($(-.5,.4)+(Origin)$);
\coordinate (Normal) at ($(.4,.9)+(Origin)$);
\coordinate (real) at ($(.19,1)+(Origin)$);
\coordinate (Ghost) at ($(.0,.73)+(Origin)$);

\coordinate (Shift) at (5,1);
\coordinate (Origin2) at ($(Origin)+(Shift)$);
\coordinate (T2) at ($(.45,.35)+(Origin2)$);
\coordinate (T12) at ($(-.5,.4)+(Origin2)$);
\coordinate (Normal2) at ($(.4,.9)+(Origin2)$);
\coordinate (Ghost2) at ($(.08,.73)+(Origin2)$);
\coordinate (T-12) at ($(.4,-.4)+(Origin2)$);

\path [draw=black, top color=gray!35,middle color=gray!60,bottom color=gray!90]($(-.3,1)+(Shift)$) arc (264.036:220.964:-3.65) -- ($(3.42,.4)+(Shift)$) arc (220.964:260.036:-3.65) -- cycle;
\filldraw[draw=black!70, fill=gray!20] (Origin2) -- (T12) -- (Ghost2) -- (T22) -- cycle;

\node at (Origin2) {{\Large\color{red}{$\bullet$}}};
\end{tikzpicture}
		\caption{Image of $\gamma_{\bb n}$}
	\end{subfigure}
	\caption{Visualisation of the trace operators.}\label{Fig::traces}
\end{figure}

Assuming compactness of $\Gamma$, we define for all $s>0$ the space $H^{-s}(\Gamma)$ as the dual space of $H^s(\Gamma).$
We define the trace space $\bb H_\times^{s}(\Gamma) \coloneqq \bb \gamma_t (\bb H^{s+ 1/2}(\Omega))$, for $0<s<1$. 
The space $\bb H^{-s}_\times(\Gamma)$ denotes the corresponding dual space w.r.t.~the duality pairing $\langle\cdot \times \bb n,\cdot\rangle_{L^2(\Gamma)}$. 
Note that $\bb H_\times^s(\Gamma)$ might not coincide with $\bb H^s(\Gamma)$ understood in a componentwise sense, since this identity holds only for smooth geometries, i.e., $C^\infty$-manifolds, see~\cite{Buffa_2002aa}.
Defining $\bb H_\times^{-1/2}(\div_\Gamma,\Gamma)\coloneqq\bb \gamma_t(\bb H^0(\bb \curl,\Omega))$ together with its rotated counterpart $\bb H_\times^{-1/2}(\curl_\Gamma,\Gamma)\coloneqq \bb\gamma_0(\bb H^0(\curl,\Omega))$, we recall the following mapping properties of the trace operators, as presented in \cite[Thm.~3.37]{McLean_2000aa} and \cite[Thm.~1, Thm.~3]{Buffa_2003aa}.

\begin{theorem}[Mapping Properties of the Trace Operators] \label{thm::htimestrace}
    For the trace operators, the following properties hold.
    \begin{enumerate}
        \item The trace operator $\gamma_0\colon H^s(\Omega)\to H^{s-1/2}(\Gamma)$ is linear, continuous and surjective, with a continuous right inverse for $0<s< 3/2$.
        \item The operator $\bb \gamma_0\colon \bb H^0(\bb \curl,\Omega)\to \bb{H}_\times^{-1/2}(\curl_\Gamma,\Gamma)$ is linear, continuous, surjective, and possesses a continuous right inverse.
        \item The operator $\bb \gamma_t\colon \bb H^0(\bb \curl,\Omega)\to \bb{H}_\times^{-1/2}(\div_\Gamma,\Gamma)$ is linear, continuous, surjective, and possesses a continuous right inverse.
        \item The operator $\gamma_{\bb n}\colon \bb H^0(\div,\Omega)\to H^{-1/2}(\Gamma)$ is linear, continuous and surjective.
    \end{enumerate}
    Moreover, for $0\leq s < 1,$ there exists a continuous extension of the tangential trace mapping $\bb \gamma_t\colon \bb H^s(\bb \curl,\Omega)\to \bb H_\times^{s-1/2}(\div_\Gamma,\Gamma)$.
\end{theorem} 

In the following, we consider a de Rham complex as in Figure \ref{fig::classicaldeRham},
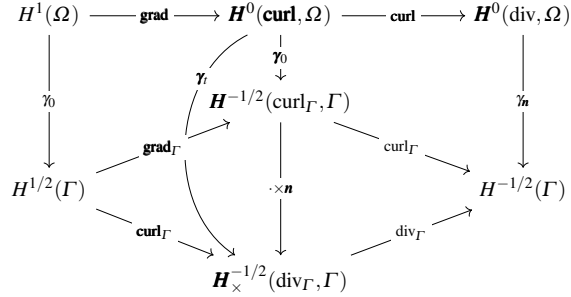
\begin{figure}
	$$\begin{tikzcd}[row sep = 2.2em,column sep = 1.4cm]
		        H^1(\Omega)\ar{dd}[description]{\gamma_0}\ar{r}[description]{\bb \grad}&
		        \bb H^0(\bb \curl,\Omega)\ar{r}[description]{\bb \curl}\ar{d}[description]{\bb \gamma_0} \arrow[ddd,bend right=60,"{\bb{\gamma}_t}" description, near start]&
		        \bb H^0(\div,\Omega)\ar{dd}[description]{\gamma_{\bb n}}\\
		        &\bb H^{{-1/2}}(\curl_\Gamma,\Gamma)\ar{dr}[description]{\curl_\Gamma}  \arrow{dd}[description]{\cdot \times \bb n} &\\
		        H^{{1/2}}(\Gamma)\ar{dr}[description]{{\vcurl_\Gamma}} \ar{ur}[description]{\bb \grad_\Gamma} & &             H^{{-1/2}}(\Gamma)\\
		         &\bb H_\times^{{-1/2}}(\div_\Gamma,\Gamma)\ar{ur}[description]{\div_\Gamma}&
		    \end{tikzcd}
		    $$
		    \caption{Two dimensional de Rham complex on the boundary, induced by application of the trace operators to the three-dimensional complex in the domain.}
\label{fig::classicaldeRham}
\end{figure}
where the trace operators map the three-dimensional spaces onto the boundary. By definition of the involved trace operators and surface differential operators, the diagram commutes.

\begin{remark}
	Note that the diagram in Figure \ref{fig::classicaldeRham} is an immensely powerful tool, showcasing the relation between the three-dimensional and two-dimension de Rham complex, and the relation of the trace spaces utilised in boundary element methods with their counterparts in the finite element context. 
	It can even be used to define the notions introduced previously: Given the trace operators $\gamma_0$, $\bb \gamma_0$ and $\gamma_{\bb n}$ as well as the three-dimensional de Rham sequence, we can {define} the trace operator $\bb \gamma_t$ by rotation around the normal and the trace spaces via the surjectivity assertions of Theorem \ref{thm::htimestrace}. Moreover, one can define the surface differential operators as the operators making the diagram commute.
\end{remark}

    As a first step towards an analysis w.r.t.~spaces of fractional regularity, we review a classical interpolation argument.
        \begin{lemma}[Interpolation Lemma]\label{interpolationlemma}
    Let $0\leq s_1\leq s_2$ and $0\leq t_1\leq t_2$ be integers and let $\Gamma$ be a compact manifold, smooth enough for the space ${H^{\max(s_2,t_2)}(\Gamma)}$ to be defined. For $\sigma\in[0,1]$, if $T \colon  H^{s_j}(\Gamma)\to H^{t_j}(\Gamma)$ is a bounded linear operator for both $j=1,2$, with
    \begin{align*}
        \norm{T u}_{H^{t_j}(\Gamma)}&\leq C_j \norm{u}_{H^{s_j}(\Gamma)},\qquad j\in \lbrace 1,2\rbrace,
    \intertext{for two constants $C_1$ and $C_2$, then we find}
        \norm{Tu}_{H^{(1-\sigma)\cdot t_1 + \sigma\cdot t_2}(\Gamma)}&\leq  C_1^{1-\sigma}C_2^\sigma\norm{u}_{H^{(1-\sigma)\cdot s_1 + \sigma\cdot s_2}(\Gamma)}.
    \end{align*}
\end{lemma}
    \begin{proof} 
        This follows by the combination of \cite[Theorem 4.1.2]{Bergh_1976aa} and \cite[Definition 2.4.1]{Bergh_1976aa}.\qed
    \end{proof}

\subsection{The Spline Complex in the Trace Space Setting}\label{subsec::complex}

    We briefly review the basic notions of isogeometric methods and refer to \cite{Cottrell_2009aa,Hughes_2005aa} for an introduction to isogeometric analysis and to \cite{Piegl_1997aa,Schumaker_2007aa} for more details on NURBS and spline theory. \nocite{Lee_1996aa}
    
\begin{definition}[B-Spline Basis {\cite[Sec.~2]{Veiga_2014aa}}]
    Let $\K$ be either $\R$ or $\mathbb C$ and $p,k$ be integers with $0\leq p< k$. We define a \emph{$p$-open knot vector} $\Xi$ as a set of knots $\xi_i$ of the form 
    \begin{align*}
        \Xi = \big\lbrace\underbrace{\xi_0 = \cdots =\xi_{p}}_{=0}< \xi_{p+1}\leq \cdots\leq \xi_{k-1} < \underbrace{\xi_{k}=\cdots =\xi_{k+p}}_{=1}\big\rbrace\in[0,1]^{k+p+1}.
    \end{align*}
    We will assume the multiplicity of interior knots to be at most $p$.
    We can then define the basis functions $ \lbrace b_i^p \rbrace_{0\leq i< k}$ on $[0,1]$ for $p=0$ as
    \begin{align*}
        b_i^0(x) & =\begin{cases}
            1, & \text{if }\xi_i\leq x<\xi_{i+1}, \\
            0, & \text{otherwise,}
        \end{cases}
        \intertext{ and for $p>0$ via the recursive relationship}
        b_i^p(x) & = \frac{x-\xi_i}{\xi_{i+p}-\xi_i}b_i^{p-1}(x) +\frac{\xi_{i+p+1}-x}{\xi_{i+p+1}-\xi_{i+1}}b_{i+1}^{p-1}(x),
    \end{align*}
    for all $0\leq i<k-1$.
    Given the basis as above, the space $S_p(\Xi)$ is given as $\operatorname{span}(\lbrace b_i^p\rbrace_{0\leq i <k})$. The integer $k$ hereby denotes the dimension of the spline space. 
\end{definition}

\begin{definition}[{\cite[Ass.~2.1]{Veiga_2014aa}}]
    For a $p$-open knot vector $\Xi,$ let $h_i\coloneqq  \xi_{i+1}-\xi_{i}.$ We define the \emph{mesh size} $h$ to be the maximal distance $h\coloneqq \max_{p\leq i < k}h_i$ between neighbouring knots.
    We call a knot vector \emph{locally quasi-uniform} when for all non-empty elements, neighbouring $[\xi_{i_1},\xi_{i_1+1}]$ and $[\xi_{i_2},\xi_{i_2+1}]$ there exists a constant $\theta\geq 1$ such that the ratio $h_{i_1}\cdot h_{i_2}^{-1}$ satisfies $\theta^{-1}\leq h_{i_1}\cdot h_{i_2}^{-1} \leq \theta.$ 
\end{definition}

Let $\ell=2,3$, and let the knot vectors $\Xi_1,\dots,\Xi_{\ell}$ be given. B-spline functions on the domain $[0,1]^{\ell}$ are constructed through simple tensor product relationships for $ p_{i_1,\dots i_\ell} \in \K$ via
\begin{align}
    f(x_1,\dots ,x_{\ell})\coloneqq\sum_{0\leq i_1< k_1}\dots\sum_{0\leq i_\ell< k_\ell}  p_{i_1,\dots,i_{\ell}} \cdot b_{i_1}^{p_1}(x_1)\cdots b_{i_{\ell}}^{p_{\ell}}(x_{\ell}),\label{def::tpspline}
\end{align}
which allows \emph{tensor product B-spline spaces}, denoted by 
	$S_{p_1,\dots,p_\ell}(\Xi_1,\dots,\Xi_\ell)$
to be defined. 
We will refer to non-empty intervals of the form $[\xi_i,\xi_{i+1}]$, $0\leq i<k$, and in the tensor product sense, non-empty sets of the form $[\xi_{i_1},\xi_{i_1+1}]\times\dots\times[\xi_{i_\ell},\xi_{i_\ell+1}]$ as \emph{elements} w.r.t.~the knot vectors.

    \begin{definition}[Support Extension, {\cite[Sec.~2.A]{Buffa_2015aa}}]
        Let $S_p({\Xi})$ be a $k$ dimensional spline space on $[0,1]$, and let
        $Q$ be an element of the knot vector $\Xi$. We define the \emph{support extension $\tilde Q$} of $Q$ by
        \begin{align*}
            \tilde Q \coloneqq \left\lbrace \textstyle{{\bigcup_{0\leq i<k}\operatorname{supp}(b_i^p)}}\colon b_i^p(x)\neq 0\text{ for }x\in Q\right\rbrace.
        \end{align*}
        The same concept is generalised by tensor product construction to spline spaces on $[0,1]^\ell$.
    \end{definition}

    \begin{assumption}[Quasi-Uniformity of Knot Vectors]\label{ass::knotvecs}
        All knot vectors will be assumed to be $p$-open and locally quasi-uniform, such that the usual spline theory is applicable \cite{Veiga_2014aa,Piegl_1997aa,Schumaker_2007aa}. 
    \end{assumption}

 	Throughout this paper, we will reserve the letter $h$ for the maximal distance between two given knots and $p$ for the minimal polynomial degree. 
 	Moreover, we let $\tilde h$ denote the maximal size of a support extension. 
    For inequalities we will use the notation 
    $$M\lesssim  T,$$
    if $M \leq C \cdot T$ holds for some constant $C>0$ independent of $h$.
    If both $M\lesssim T$ and $T\lesssim M$ hold, we will write $M\simeq T$.

	\begin{definition}[Patch] \label{def:patch}
		We define a \emph{patch} $\Gamma$ to be the image of $[0,1]^2$ under an invertible diffeomorphism $\bb F\colon [0,1]^2\to\Gamma\subseteq\mathbb R^3$.
		Let $\Omega$ be a Lipschitz domain. We define a \emph{multipatch geometry} to be a compact, orientable two-dimensional manifold $\Gamma=\partial \Omega$ invoked via  $\bigcup_{0\leq j <N} \Gamma_j$ by a family of patches $\lbrace \Gamma_j\rbrace_{0\leq j<N},$ $N\in \mathbb N$, given by a family of diffeomorphisms $$\lbrace \bb F_j \colon [0,1]^2\hookrightarrow \Gamma_j\rbrace_{0\leq j<N},$$ called \emph{parametrisation}.
		We require the images of $(0,1)^2$ of all $\bb F_j$ to be disjoint and that for any \emph{patch interface} $D$ of the form $D=\partial\Gamma_{j_0}\cap \partial\Gamma_{j_1}\neq \emptyset$, we find that the parametrisations $\bb F_{j_0}$ and $\bb F_{j_1}$ coincide.
	\end{definition}

    Note that this definition excludes non-watertight geometries and geometries with T-junctions, since mappings at interfaces must coincide, cf.~Figure \ref{Fig::Multipatch}.

    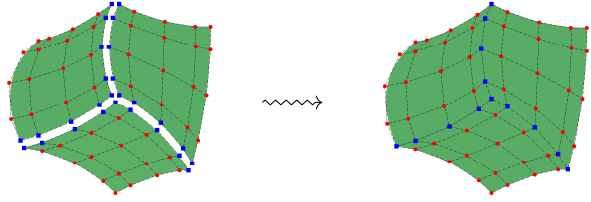
\begin{figure}
    	\centering
    	\begin{tikzpicture}
    		\node at (0,0) {\includegraphics[width=.3\textwidth]{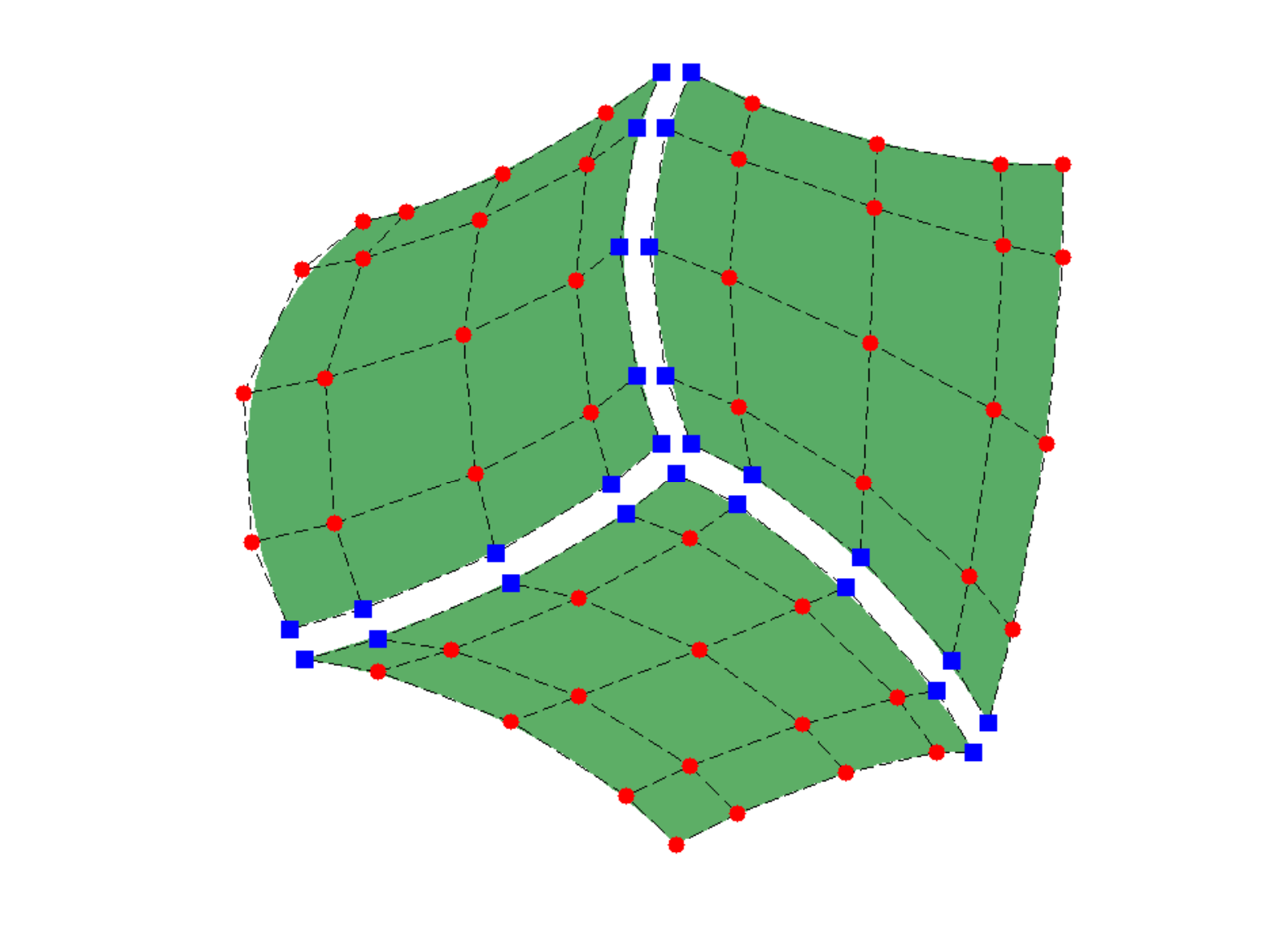}};
    		\node at (5,0) {\includegraphics[width=.3\textwidth]{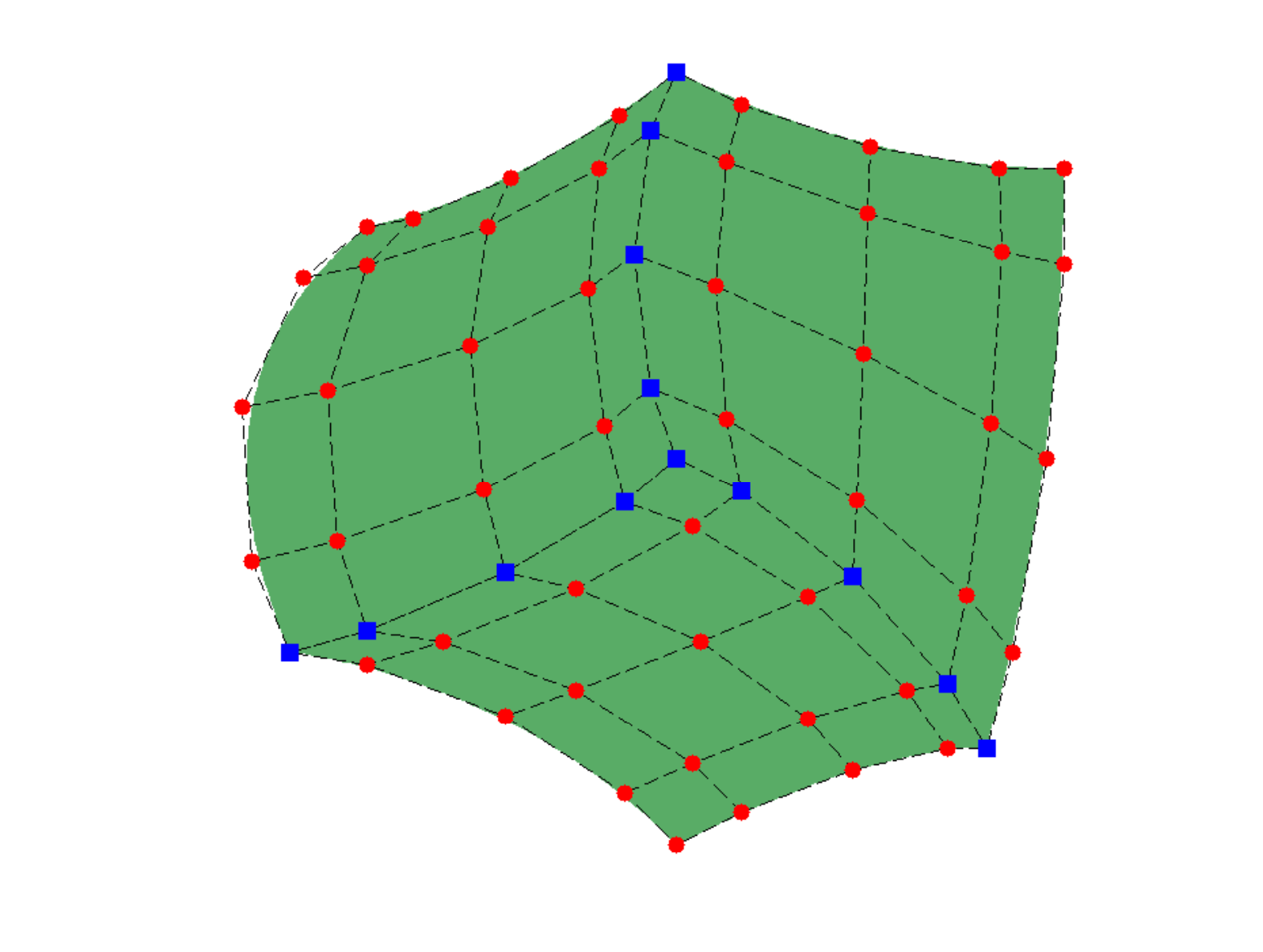}};
    		\node (A) at (2,0) {};
    		\node (B) at (3,0) {};
    		\draw [->,
line join=round,
decorate, decoration={
    zigzag,
    segment length=4,
    amplitude=.9,post=lineto,
    post length=2pt
}]  (A) -- (B);
    	\end{tikzpicture}
    	\caption{Mappings on interfaces must coincide.}
    	\label{Fig::Multipatch}
    \end{figure}

    In the spirit of isogeometric analysis, these mappings will usually be given by NURBS mappings, i.e.,~by
    \begin{align*}
        \bb F_j(x,y)\coloneqq \sum_{0\leq j_1<k_1}\sum_{0\leq j_2<k_2}\frac{\bb c_{j_1,j_2} b_{j_1}^{p_1}(x) b_{j_2}^{p_2}(y) w_{j_1,j_2}}{ \sum_{i_1=0}^{k_1-1}\sum_{i_2=0}^{k_2-1} b_{i_1}^{p_1}(x) b_{i_2}^{p_2}(y) w_{i_1,i_2}},
    \end{align*}
    for control points $\bb c_{j_1,j_2}\in \R^3$ and weights $w_{i_1,i_2}>0.$ In accordance with the isogeometric framework, degrees and knot vectors of the discrete spaces to be mapped from the reference domain are usually chosen in accordance with the parametrisation \cite{Hughes_2005aa}.
    However, the description of the geometry is, in principle, independent of the analysis that will follow. 
    From now on we reserve the letter $N$ for the number of patches and the letter $j$ to refer to a generic patch.

    As NURBS with interior knot repetition are not arbitrarily smooth, one would usually resort to the utilisation of bent Sobolev spaces \cite{Veiga_2014aa}. However, to avoid technical details, we introduce the following assumption.
    \begin{assumption}[Smoothness of Geometry Mappings]\label{ass::geometry}
        We assume any multipatch geometry to be given by an invertible, non-singular parametrisation $\lbrace \bb F_j\rbrace_{0\leq j<N}$ with $\bb F_j\in C^\infty([0,1]^2).$
    \end{assumption} 

   {We remark that Assumption~\ref{ass::geometry} implies that each patch $\Gamma_j$ has Lipschitz boundary.}
    We {also} stress that, limited by the smoothness of $\lbrace \bb F_j\rbrace_{0\leq j< N}$, all results are still provable for non-smooth but invertible NURBS parametrisation, although this would require an analysis via bent Sobolev spaces as in \cite{Veiga_2014aa}. Assumption \ref{ass::geometry} is merely for convenience. Moreover, it is possible to obtain parametric mappings satisfying Assumption \ref{ass::geometry} either through extraction of rational Bézier patches, which can be obtained as subpatches of a given NURBS parametrisation or, more generally, through an algorithmic approach as in \cite{Harbrecht_2010aa}.

    \begin{definition}[Spaces of Patchwise Regularity]
        Let $\Gamma = \bigcup_{0\leq j< N} \Gamma_j$ be a multipatch geometry. We define the norm
        \begin{align*}
            \norm{f}_{H_{{\mathrm{pw}}}^s(\Gamma)}^2 \coloneqq \sum_{0\leq j<N} \norm{f|_{\Gamma_j}}_{H^s(\Gamma_j)}^2
        \end{align*}
        for all $f\in L^2(\Gamma)$ for which the right-hand side is well defined, and
        define the corresponding space equipped with this norm as
        \begin{align*}
            H_{{\mathrm{pw}}}^s(\Gamma)\coloneqq \lbrace f\in L^2(\Gamma)\colon \norm{f}_{H_{{\mathrm{pw}}}^s(\Gamma)} < \infty \rbrace.
        \end{align*}
        In complete analogy, we extend the definition to vector-valued Sobolev spaces (and spaces with graph norms), as usual, denoted by bold letters $  {\bb H}_{{\mathrm{pw}}}^s(\Gamma)$.
    \end{definition}

    \begin{definition}[Single Patch Spline Complex \cite{Buffa_2011aa}] \label{def:spline-complex}
        Let $\bb p=(p_1,p_2)$ be a pair of positive integers and $\Xi_1,\Xi_2$  be $p$-open knot vectors on $[0,1].$ Let $\Xi_1'$ and $\Xi_2'$ denote their truncation, i.e., the knot vector without its first and last knot. We define the \emph{spline complex} on $[0,1]^2$ as the spaces 
        \begin{align*}
            \S^0_{\bb p,\bb\Xi}([0,1]^2)\coloneqq {}&{} S_{p_1,p_2}(\Xi_1,\Xi_2),\\
            \bb \S^1_{\bb p,\bb\Xi}([0,1]^2)\coloneqq {}&{}  S_{p_1,p_2-1}(\Xi_1,\Xi_2') \times  S_{p_1-1,p_2}(\Xi_1',\Xi_2),\\
            \S^2_{\bb p,\bb\Xi}([0,1]^2)\coloneqq {}&{} S_{p_1-1,p_2-1}(\Xi_1',\Xi_2').
        \end{align*}
    \end{definition}
    	In the reference domain, the spline complex can be visualised as in Figure \ref{Fig::Complex}.
    	\begin{figure}
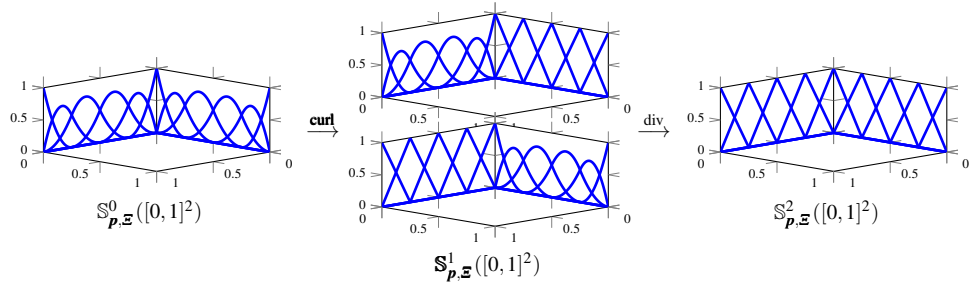

    		\begin{tikzpicture}
	\node (center) {};
	\node[left of = center,node distance = 4.5cm] (a) {\begin{tiny}
	\input{pics/fig1}\end{tiny}};
	\node[above of =center, node distance = .73cm] {\begin{tiny}
	\input{pics/fig2b}\end{tiny}};
	\node[below of =center, node distance = .73cm] {\begin{tiny}
	\input{pics/fig2a}\end{tiny}};
	\node[right of =center, node distance = 4.5cm] (b) {\begin{tiny}
	\input{pics/fig3}\end{tiny}};
	\node[left of = center, node distance = 2.2cm] {$\stackrel{\bb{\curl}}{\longrightarrow}$};
	\node[right of = center, node distance = 2.2cm] {$\stackrel{\div}{\longrightarrow}$};
	\node[below of =a,node distance = 1.2cm]{$\S^0_{\bb p,\bb \Xi}([0,1]^2)$};
	\node[below of =center,node distance = 1.9cm]{$\bb\S^1_{\bb p,\bb \Xi}([0,1]^2)$};
	\node[below of =b,node distance = 1.2cm]{$\S^2_{\bb p,\bb \Xi}([0,1]^2)$};
\end{tikzpicture}
    		\caption{Visualisation of the single patch spline complex for $\bb p= (2,2)$. {The blue functions are the univariate B-splines related to each coordinate direction, whose  tensor-product gives the bases of the spline spaces in Definition~\ref{def:spline-complex}.} {The two-dimensional $\vcurl$ operator maps the smooth space to a vector valued space where the regularity in each vector component is lowered w.r.t.~one spatial component. Analogously, the divergence operator maps to the space of globally lowered regularity.}}\label{Fig::Complex}
    	\end{figure}
        Assume $\Gamma$ to be a single patch domain given via a geometry mapping $\bb F$ in accordance with Assumption \ref{ass::geometry}. 
        To define the spaces in the physical domain, we resort to an application of the pull-backs, which, as a study of \cite{Peterson_2006aa} reveals, are given by
        \begin{align*}
            \iota_0(\bb F)(f_0)(\bb x)&\coloneqq \big(f_0\circ\bb F\big)(\bb x),&&\bb x\in [0,1]^2,\\
            \iota_1(\bb F)(\bb f_1)(\bb x)&\coloneqq\big( \kappa(\bb x)\cdot (d\bb F)^{-1} (\bb f_1\circ \bb F)\big)(\bb x),&&\bb x\in [0,1]^2,\\
            \iota_2(\bb F)(f_2)(\bb x)&\coloneqq \big(\kappa(\bb x)\cdot (f_2\circ \bb F)\big)(\bb x),&&\bb x\in [0,1]^2,
        \end{align*}
        where the term $\kappa$ for $\bb x\in[0,1]^2$ is given by the so-called \emph{surface measure}
        \begin{align}
            \kappa(\bb x)\coloneqq \norm{\partial_{x_1}\bb F( {\bb x})\times \partial_{x_2}\bb F( {\bb x})}.
        \end{align}
        Note that if one were to compute the pullbacks $\iota_i(\bb F)$ for $i=0,1,2$ as above, at first glance one were to encounter a dimensionality problem, since the inverse $d\bb F^{-1}$ of the Jacobian $d\bb F$ arising from $\bb F$ is of size $2\times 3$, and thus not readily invertible. 
        The study of e.g.~\cite{Bossavit_1988aa,Dhaeseleer_1991aa,Kurz_2007aa} makes it clear that, due to Assumption \ref{ass::geometry}, required inverse mappings for the case of embedding a two-dimensional manifold into three-dimensional ambient space exist. They need to be understood as mappings from $[0,1]^2$ onto the tangential space of $\Gamma$. It is merely a smooth one to one mapping between a two-dimensional space into another, and invertibility must be understood in this sense.
        However, for implementation this matters little, since both ansatz- and test functions will be defined on $[0,1]^2$. Therefore one merely needs to compute the corresponding push-forwards, readily available through the equalities
         \begin{align*}
            (\iota_0(\bb F))^{-1}(f_0)&= \big(f_0\circ \bb F^{-1}\big)(\bb x),&&\bb x\in \Gamma,\\
            (\iota_1(\bb F))^{-1}(\bb f_1)&= \big({\kappa(\bb x)^{-1}} \cdot (d\bb F) \bb f_1\circ \bb F^{-1}\big)(\bb x),&&\bb x\in \Gamma,\\
            (\iota_2(\bb F))^{-1}(f_2)&= \big({\kappa(\bb x)^{-1}} \cdot f_2\circ \bb F^{-1}\big)(\bb x),&&\bb x\in \Gamma,
        \end{align*}
        due to Assumption \ref{ass::geometry}. The inverse of $\bb F$ needs not be computed, since pull-backs and push-forwards cancel out by construction.

        \begin{remark}\label{rem::stillconforming}
            A study of \cite[{Chap.~5}]{Peterson_1995aa} makes clear that these mappings are still conforming for $\Gamma_j\subseteq\mathbb R^3,$ i.e.,~that the diagram   
            $$ \begin{tikzcd}[row sep = 3em,column sep = 1.3cm]
            H^1\big((0,1)^2\big)\ar{d}[description]{\iota_0(\bb F_j)^{-1}}\ar{r}[description]{\vcurl}& \bb H\big(\div,(0,1)^2\big)\ar{r}[description]{\div}\ar{d}[description]{\iota_1(\bb F_j)^{-1}}&  L^2\big((0,1)^2\big)\ar{d}[description]{\iota_2(\bb F_j)^{-1}}\\
            H^{1}(\Gamma_j)\ar{r}[description]{{\vcurl_\Gamma}} & \bb H(\div_\Gamma,\Gamma_j)\ar{r}[description]{\div_\Gamma} &  L^{{2}}(\Gamma_j)
        \end{tikzcd}$$
            commutes. Because of this, we can identify the divergence on the reference domain with the divergence on the physical domain, up to a bounded factor induced by the corresponding pull-back, due to Assumption \ref{ass::geometry}. We will, later on, utilise this explicitly to apply estimates of the kind
            \begin{align*}
                \norm{\div_\Gamma f}_{L^2(\Gamma_j)} &\simeq \norm{\div( f\circ \bb F)}_{L^2([0,1]^2)},
            \end{align*}
            see also \cite{Monk_2003aa,Peterson_2006aa} for a further review of these concepts.
        \end{remark}

        Now we can define corresponding discretisations on the physical domain $\Gamma_j$ by
        \begin{align}
        \begin{aligned}
            \S^0_{\bb p,\bb\Xi}(\Gamma_j)\coloneqq {}&{} \left\lbrace f  \colon \iota_0(\bb F_j)(f) \in \S^0_{\bb p,\bb\Xi}([0,1]^2)\right\rbrace,\\
            \bb \S^1_{\bb p,\bb\Xi}(\Gamma_j)\coloneqq {}&{} \left\lbrace \bb f  \colon \iota_1(\bb F_j)(\bb f) \in \bb \S_{\bb p,\bb\Xi}^1([0,1]^2)\right\rbrace,\\
            \S^2_{\bb p,\bb\Xi}(\Gamma_j)\coloneqq {}&{} \left\lbrace f  \colon \iota_2(\bb F_j)(f) \in \S^2_{\bb p,\bb\Xi}([0,1]^2)\right\rbrace.
            \end{aligned}\label{def::singlepatchphysical}
         \end{align} 
    Proceeding as in \cite{Veiga_2014aa} the spline complex for spaces on the boundary is defined as follows.

    \begin{definition}[Multipatch Spline Complex on Trace Spaces]
        Let $\Gamma = \bigcup_{0\leq j< N} \Gamma_j$ be a multipatch boundary satisfying Assumption \ref{ass::geometry}. Moreover, let $\bb \Xi \coloneqq (\bb \Xi_j)_{0\leq j<N}$ be pairs of knot vectors in accordance with Assumption \ref{ass::knotvecs} and $\bb p=(\bb p_j)_{0\leq j<N}$ pairs of integers, corresponding to polynomial degrees. Then we define the \emph{spline complex on the boundary} $\Gamma$ via
        \begin{align*}
            \S^0_{\bb p,\bb\Xi}(\Gamma)\coloneqq {}&{} \left\lbrace f\in H^{1/2}(\Gamma)\colon  f|_{\Gamma_j} \in \S^0_{\bb p_j,\bb\Xi_j}(\bb \Gamma_j)\text{ for all }0 \leq j<N\right\rbrace,\\
             \bb \S_{\bb p,\bb\Xi}^1(\Gamma)\coloneqq {}&{} \left\lbrace \bb f\in \bb H_\times^{{-1/2}}(\div_\Gamma,\Gamma)\colon  \bb f|_{\Gamma_j} \in \bb \S_{\bb p_j,\bb\Xi_j}^1(\bb \Gamma_j)\text{ for all }0\leq j< N\right\rbrace,\\
             \S^2_{\bb p,\bb\Xi}(\Gamma)\coloneqq {}&{} \left\lbrace f\in H^{-1/2}(\Gamma)\colon  f|_{\Gamma_j} \in \S^2_{\bb p_j,\bb\Xi_j}(\bb \Gamma_j)\text{ for all }0\leq j< N\right\rbrace.
        \end{align*}
        We assume $\bb p$ and $\bb \Xi$ to be such that they coincide on every patch-interface.
    \end{definition}


    \begin{remark}\label{rem::defviatrace}
        Note that a different definition of the considered spline spaces could be achieved by application of the trace operators to the volumetric parametrisation, provided their existence, see Theorem \ref{thm::htimestrace}. However, the construction above seems more suitable for the analysis of approximation properties.
    \end{remark}

    \section{Approximation Properties of Conforming Spline Spaces} \label{subsec::approx}
    We will now investigate approximation properties of the spaces defined in the previous section. This will be done through the introduction of quasi-interpolation operators, projections, which are defined in terms of a dual basis.

    For one-dimensional spline spaces Schumaker \cite[Sec.~4.2]{Schumaker_2007aa} introduced quasi-interpolants, defined in via some dual functionals $$\lambda_{i,p}\colon L^2([\xi_i,\xi_{i+p+1}])\to \mathbb K,$$ such that
    \begin{align}
        \Pi_{p,\Xi}\colon f\mapsto\sum_{0\leq i < k} \lambda_{i,p}(f) b_{i}^p.
    \end{align}
    Note that, by definition of the $\lambda_{i,p}$ they merely require $f$ to be square integrable. Moreover, the operators depend on the specific knot vectors, which we do not reference for notational purposes.

    As shown in \cite{Veiga_2014aa}, a tensor product construction utilising the above projection yields interpolants $\Pi_{\bb p,\bb \Xi}^0,$ $\bb{\Pi}_{\bb p,\bb \Xi}^1,$ $\Pi_{\bb p,\bb \Xi}^2$ mapping onto the spaces $\S^0_{\bb p,\bb \Xi}([0,1]^2),$  $\bb \S_{\bb p,\bb\Xi}^1([0,1]^2)$, and $\S^2_{\bb p,\bb \Xi}([0,1]^2),$ as explained in \cite[p.~169ff]{Veiga_2014aa}, where error estimations and $L^2$-stability for B-spline approximations have been also been provided.
    A crucial property of the construction is as follows.
    \begin{lemma}[Commuting Interpolation Operators,~{\cite[Prop.~5.8]{Veiga_2014aa}}]\label{lem::theycommute} 
        The diagram
        $$
            \begin{tikzcd}
                H^{1}\big((0,1)^2\big)\ar{r}{\vcurl}\ar{d}{\Pi_{\bb p,\bb \Xi}^0}&\bb H\big(\div,(0,1)^2\big)\ar{r}{\div}\ar{d}{\bb\Pi_{\bb p,\bb \Xi}^{1}}&L^2\big((0,1)^2\big)\ar{d}{\Pi_{\bb p,\bb \Xi}^2}\\
                \S_{\bb p,\bb \Xi}^0([0,1]^2)\ar{r}{\vcurl }&\bb\S_{\bb p,\bb \Xi}^{1}([0,1]^2)\ar{r}{\div}& \S_{\bb p,\bb \Xi}^2([0,1]^2)
            \end{tikzcd}
        $$
        commutes.
    \end{lemma}
    \begin{remark}
        For the two-dimensional setting \cite{Veiga_2014aa} introduces two spaces $\bb \S_{\bb p,\bb\Xi}^1$ and $\bb \S_{\bb p,\bb\Xi}^{1*}$, which correspond to curl conforming and divergence conforming spaces, respectively. Since we are interested mostly in spaces of the $\div$-type and the spaces differ only by a rotation, we will not mention the two different types of spline spaces. However, it should be noted that our spaces of type $\bb \S_{\bb p,\bb\Xi}^1$ correspond to those of type $\bb \S_{\bb p,\bb\Xi}^{1*}$ in the cited literature.
    \end{remark}
    By application of the pull-backs used to define the spline spaces one can immediately generalise the projectors and all results to the case of functions on the physical domains.
    Corollary 5.12 of \cite{Veiga_2014aa} reveals that for the case of a single patch $\Gamma_j$ the following holds.
    \begin{corollary}[Single Patch Approximation Estimate,~{\cite[Cor.~5.12]{Veiga_2014aa}}]\label{cor::Approxcor}
    Let $\Gamma_j$ be a single patch domain and let Assumptions \ref{ass::knotvecs} and \ref{ass::geometry} hold. Then we find that
    \begin{align*}
            \norm{u -\Pi^0_{\bb p,\bb \Xi}u}_{H^r(\Gamma_j)}
            {} & {}\lesssim h^{s-r}\norm{u}_{H^s(\Gamma_j)},\qquad0\leq r \leq s\leq p+1,       \\
            \norm{\bb u -\bb \Pi_{\bb p,\bb \Xi}^1\bb u}_{\bb H^r(\Gamma_j)}
            {} & {}\lesssim h^{s-r}\norm{\bb u}_{\bb H^s(\Gamma_j)},\qquad0\leq r \leq s\leq p, \\
            \norm{u -\Pi_{\bb p,\bb \Xi}^{2} u}_{H^r(\Gamma_j)}
            {} & {}\lesssim h^{s-r}\norm{ u}_{H^s(\Gamma_j)},\qquad0\leq r \leq s\leq p. \\
    \end{align*}
    \end{corollary}

    Indeed, the construction of $\bb\Pi^1_{\bb p,\bb \Xi}$ makes it possible to estimate 
        $$\norm{\bb u -\bb \Pi^1_{\bb p,\bb \Xi}\bb u}_{\bb H^r(\div_\Gamma,\Gamma_j)}\lesssim h^{s-r}\norm{\bb u}_{\bb H^s(\div_\Gamma,\Gamma_j)},\qquad0\leq r \leq s\leq p,$$
    since, by properties of the pull-backs, the operators also commute w.r.t.~the surface differential operators {one finds that
    \begin{align*}
      \norm{\bb u -\bb \Pi^1_{\bb p,\bb \Xi}\bb u}^2_{\bb H^r(\div_\Gamma,\Gamma_j)} ={}&{}  \norm{\bb u -\bb \Pi^1_{\bb p,\bb \Xi}\bb u}^2_{\bb H^r(\Gamma_j)} + \norm{\div_\Gamma(\bb u) -\div_\Gamma(\bb\Pi^1_{\bb p,\bb \Xi} \bb u)}^2_{  H^r(\Gamma_j)}\\
      ={}&{}  \norm{\bb u -\bb \Pi^1_{\bb p,\bb \Xi}\bb u}^2_{\bb H^r(\Gamma_j)} + \norm{\div_\Gamma(\bb u) -\Pi^2_{\bb p,\bb \Xi} \div_\Gamma(\bb u)}^2_{  H^r(\Gamma_j)},
    \end{align*} 
    allowing to apply the estimates of the previous corollary.
    }

    For the remainder of this section, we will generalise these notions for the multipatch case.

    \subsection{Multipatch Quasi-interpolation Operators}

    We now want to generalise the above to the multipatch setting. {For this, we need to construct interpolation operators capable of preserving continuity across patch boundaries.}
    For one-dimensional spline spaces $S^p(\Xi)$ and $f\in C^\infty([0,1]),$ \cite[{Sec.~2.1.5}]{Veiga_2014aa} defines
    $$    \tilde \Pi_{p,\Xi}\colon f \mapsto \sum_{0\leq  i < k}\tilde \lambda_{i,p}(f)b_{i}^{p},$$ where for $0< i< k-1$ we set
        $\tilde \lambda_{i,p}(f)=\lambda_{i,p}(f),$
    but additionally, require
    \begin{align*}
        \tilde \lambda_{0,p}(f)=f(0)\qquad\text{as well as}\qquad
        \tilde \lambda_{k-1,p}(f)=f(1).
    \end{align*}
    This will yield versions of the projection operators which respect boundary conditions.

   Analogously to the construction in \cite{Buffa_2011aa}, we can now construct quasi-interpolation operators for the multipatch case that commute w.r.t.~derivation. Investigation of the one-dimensional diagram
    \begin{equation}
        \begin{tikzcd}
            H^1(0,1) \ar{d}{ \tilde \Pi_{p,\Xi}} \ar{r}{\partial_x} & \ar{d}{\tilde \Pi_{p,\Xi}^\partial} L^2(0,1)\\
            S^p(\Xi) \ar{r}{\partial_x}&S^{p-1}(\Xi')
        \end{tikzcd}\label{diag::1d}
    \end{equation}
    makes clear that a suitable choice of $\tilde \Pi_{p,\Xi}^\partial$ is given by
    \begin{align}
        \tilde \Pi_{p,\Xi}^\partial \colon f \mapsto \partial_\eta \left[\tilde \Pi_{p,\Xi} \int_0^\eta f(x) \opd x\right]. \label{def::tildeop}
    \end{align}
    {By diagram chase and application of the fundamental theorem of calculus one can see that} \eqref{def::tildeop} renders Diagram \eqref{diag::1d} commutative.
    
    \begin{proposition}[Spline Preserving Property]
    	The operator $\tilde\Pi^\partial_{p,\Xi}\colon L^2(0,1)\to S^{p-1}(\Xi')$ preserves B-splines within $S^{p-1}(\Xi')$.
    \end{proposition}
    \smartqed
    \begin{proof} 
    		By \cite[{Sec.~2}]{Buffa_2015aa} we know that the assertion holds for $\tilde \Pi_{p,\Xi}$.
    		Fixing a spline $b'\in S^{p-1}(\Xi'),$ we know that there exists a $b\in S^{p}(\Xi)$ with $\partial_x b = b'$, since $\partial_x\colon S^{p}(\Xi) \to S^{p-1}(\Xi')$ is surjective. Now, since $b\in H^1(0,1),$ the assertion follows by diagram chase.\qed
    	\end{proof}
      An immediate consequence of this proposition is the fact, that the operator ${\tilde\Pi}^\partial_{p,\Xi}$ is a projection.

    Defining quasi-interpolation operators via
    \begin{align}
        \begin{aligned}
            \tilde \Pi_{\bb p,\bb \Xi}^0   & \coloneqq \tilde \Pi_{p_1,\Xi_1}\otimes \tilde \Pi_{p_2,\Xi_2}, \\
            \tilde {\bb\Pi}_{\bb p,\bb \Xi}^{1} & \coloneqq (\tilde \Pi_{p_1,\Xi_1}\otimes \tilde \Pi_{p_2,\Xi_2}^\partial)\times (\tilde \Pi_{p_1,\Xi_1}^\partial\otimes \tilde \Pi_{p_2,\Xi_2}),  \\
            \tilde \Pi_{\bb p,\bb \Xi}^2   & \coloneqq \tilde \Pi_{p_1,\Xi_1}^\partial\otimes \tilde \Pi_{p_2,\Xi_2}^\partial, \label{def::commtilde}
        \end{aligned}
    \end{align}
    {where $\otimes$ denotes the tensor-product and $\times$ denotes the Cartesian product,} we can now define global projections on the physical domain via application of the pull-backs.

    \begin{definition}[Global Interpolation Operators]\label{def::globalinterpolants}
    Let $\bb \Xi$ and $\bb p$ denote $N$-tuples of pairs of knot vectors and polynomial degrees, respectively. Let $\Gamma=\bigcup_{0\leq j<N}\Gamma_j$ be a multipatch boundary induced by a family of diffeomorphisms $\lbrace \bb F_j\rbrace_{0\leq j< N}$ {as in Definition~\ref{def:patch}}.
    {For a family of patchwise linear operators $\lbrace L_j\rbrace_{0\leq j< N}$ we denote by $\lbrace \tilde L_j\rbrace_{0\leq j< N}$ their extensions by 0 onto $\Gamma $ and write 
                $$\bigoplus_{0\leq j< N} L_j \coloneqq \sum_{0\leq j< N} \tilde L_j.$$}

    Now, the global B-spline projections are defined as
            \begin{align*}
                \tilde\Pi_\Gamma^0\coloneqq {}&{}\bigoplus_{0\leq j< N}  \left((\iota_{0}(\bb F_j))^{-1}\circ\tilde\Pi_{\bb p_j,\bb\Xi_j}^0\circ\iota_{0}(\bb F_j)\right),\\
                \tilde{\bb\Pi}_\Gamma^1\coloneqq {}&{}\bigoplus_{0\leq j< N}  \left((\iota_{1}(\bb F_j))^{-1}\circ\tilde{\bb\Pi}_{\bb p_j,\bb\Xi_j}^1\circ\iota_{1}(\bb F_j)\right),\\
                \tilde\Pi_\Gamma^2\coloneqq {}&{}\bigoplus_{0\leq j< N}  \left((\iota_{2}(\bb F_j))^{-1}\circ\tilde\Pi_{\bb p_j,\bb\Xi_j}^2\circ\iota_{2}(\bb F_j)\right),
            \end{align*}
            i.e.,~by patchwise application of the projections of \eqref{def::commtilde} with their corresponding pull-backs and push-forwards. 
    \end{definition}
    Note that, since the pullbacks are commuting with the differential operators in the reference domain and surface differential operators, an analogue of Lemma \ref{lem::theycommute} holds also for the global interpolants \cite{Peterson_2006aa}. 

    For the global interpolation operators to be well defined, we require a certain amount of regularity. This can be formalised as follows.

     \begin{lemma}[Regularity Required for the Commuting Diagram Property]\label{lem::multipatch::theycommute::and::regularity}
    The interpolation operators $\tilde\Pi_\Gamma^{0},$ $\tilde{\bb\Pi}_\Gamma^{1}$ and $\tilde\Pi_\Gamma^{2}$ are well defined for functions in  $H^{1+\varepsilon}(\Gamma)$, ${\bb H}^{\varepsilon}(\div_\Gamma,\Gamma)$ and $H^\varepsilon(\Gamma),$ respectively, for any $\varepsilon>0$.
    Moreover, the diagram
    $$
        \begin{tikzcd}
            H^{1+\varepsilon}(\Gamma)\ar{r}{\vcurl_\Gamma}\ar{d}{\tilde\Pi_\Gamma^0}&\bb H^{\varepsilon}(\div_\Gamma,\Gamma)\ar{r}{\div_\Gamma}\ar{d}{\tilde{\bb\Pi}_\Gamma^{1}}&H^\varepsilon(\Gamma)\ar{d}{\tilde\Pi_\Gamma^2}\\
            \S_{\bb p,\bb\Xi}^0(\Gamma)\ar{r}{\vcurl_\Gamma }&\bb\S_{\bb p,\bb\Xi}^{1}(\Gamma)\ar{r}{\div_\Gamma}&\S_{\bb p,\bb\Xi}^2(\Gamma)
        \end{tikzcd}
    $$
    commutes.
    \end{lemma}
    \begin{proof} 
    	By Sobolev Imbedding Theorems, see \cite[Sec.~8]{Nezza_2012aa}, we know that any function in $H^{1+\varepsilon}(\Gamma)$ admits a continuous representative. Thus, by definition of $\tilde{\Pi}^0_\Gamma,$ it is well defined for functions in $H^{1+\varepsilon}(\Gamma)$.
    	Its definition via integration makes the operators $\tilde \Pi_{p,\Xi}^\partial$ well defined for functions in $H^\varepsilon(\Gamma)$, which immediately yields the assertion about $\tilde\Pi^2_\Gamma.$ \nocite{Adams_1978aa,wloka_1987aa}
    	It remains to show that $\bb H^\varepsilon(\div_\Gamma,\Gamma)$ is within the domain of $\tilde {\bb\Pi}^1_\Gamma.$ 
    	Considering each interface separately, by applying Gauss' theorem in the union of the two adjacent patches, one can see that the normal component of any function in $\bb H^\varepsilon(\div_\Gamma,\Gamma)$ is continuous across patch boundaries. By tensor product construction of $\bb{\tilde \Pi^1}_{\Gamma}$
    	on each patch w.r.t.~the reference domain, the continuous  component can be identified with the domain of the operators of type $\tilde \Pi_{p,\Xi}$. Thus, the interpolation is well defined. 
    	With regard to the interior and the tangential component along patch interfaces, by definition of the dual functionals $\tilde \lambda_i$ via integration, and integration within the definition of $\tilde \Pi^\partial_{p,\Xi}$, regularity of $\bb H^\varepsilon(\Gamma)$ suffices for the operation to be well defined. The commuting property {follows analogously to \cite[Prop.~5.8]{Veiga_2014aa}}.\qed
    \end{proof}

    The constructions of \eqref{def::commtilde} and Definition \ref{def::globalinterpolants} can easily be generalised to three dimensions, see \appendA.

    \subsection{Convergence Properties of Multipatch Quasi-Interpolation Operators}
    We will now provide approximation estimates for the introduced interpolation operators. Note that,
    by construction, it is clear that the boundary interpolating projections commute w.r.t.~the differential operators. It is however not clear whether the construction in \eqref{def::tildeop} and \eqref{def::commtilde} impacts the convergence behaviour w.r.t.~$h$-refinement.
      
    To utilise the commuting property to show convergence in the energy spaces, we need an analogue of Corollary \ref{cor::Approxcor} for the multipatch operators. 

    The classical proofs rely heavily on the $L^2$-stability of the projectors. Unfortunately, due to the interpolation at $0$ and $1$, the multipatch variants lose this property. Thus, we need to establish another suitable stability condition.

    \begin{proposition}[Stability of $\tilde \Pi_{p,\Xi}$]\label{Prop::TildeStability}
        Let Assumption \ref{ass::knotvecs} hold.
        Assume $f$ to be continuous in a neighbourhood around 0 and 1 and let $I = (\xi_j,\xi_{j+1})$. Let $\tilde I$ denote the support extension of $I$. Then it holds that 
        \begin{align}
            \norm{\tilde\Pi_{p,\Xi} (f)}_{L^2(I)}\lesssim{} &{}   \norm{f}_{L^2(\tilde I)}+h\abs{f}_{H^1(\tilde I)} ,\label{eq::notquiteL2stable}\\
            \abs{\tilde \Pi_{p,\Xi} (f)}_{H^1(I)}\lesssim {}&{}  \norm{f}_{H^1(\tilde I)}.\label{eq::secondStability}
        \end{align}
        Moreover, we find
        \begin{align*}
            \norm{{\tilde\Pi}_{p,\Xi}^\partial(f)}_{L^2(I)} \lesssim   \norm{f}_{L^2(\tilde I)}.
        \end{align*}
\end{proposition}
        \begin{proof} 
            The first two inequalities have been discussed by \cite[{Prop.~2.3}]{Buffa_2015aa}.
            Investigating the third assertion, we set $g(x)=\int_0^xf(t)\opd t$. 

        The proof concludes by a nontrivial application of the Poincaré inequality as follows. 

        For this, we set $C=-\frac{1}{\abs{\tilde I }}\int_{\tilde I} g \opd x$, where $\abs{\tilde I}$ denotes the Lebesgue measure of $\tilde I$, and observe that 
        \begin{align}
        	\norm{{\tilde\Pi}^\partial_{p,\Xi}(f)}_{L^2(I)}
        	= {}&{}\norm{\partial_x\tilde\Pi_{p,\Xi} \int_0^xf(t)\opd t }_{L^2(I)}\notag\\
        	= {}&{}\norm{\partial_x\tilde \Pi_{p,\Xi}\left(\int_0^x f(t)\opd t +C\right)}_{L^2(I)}\notag\\
        	= {}&{}\abs{\tilde\Pi_{p,\Xi}\left(\int_0^x f(t)\opd t +C\right)}_{H^1(I)}\notag\\
        	\lesssim {}&{}(\norm{g+C}_{L^2(\tilde I)}^2 + \abs{g + C}_{H^1(\tilde I)}^2)^{1/2},\label{eq::followsfrom}
        \end{align}
        where the inequality follows from \eqref{eq::secondStability}.
        Now, since by definition of $C$ we find that $\frac{1}{\abs{\tilde I }}\int_{\tilde I}g\opd x=-C$, we can apply the Poincaré inequality, see e.g.~\cite{wloka_1987aa}, which yields
        $$\norm{g + C}_{L^2(\tilde I )}\lesssim 
        \abs{g}_{H^1(\tilde I)} = \norm{f}_{L^2(\tilde I)},$$ 
        for the first term of \eqref{eq::followsfrom}. For the second term, we find $$\abs{g+C}_{H^1(\tilde I)}^2= \int_{\tilde I}\abs{\partial_x(g(x)+C)}^2\opd x = \abs{g}_{H^1(\tilde I)}^2= \norm{f}_{L^2(\tilde I)}^2$$ and the assertion follows.\qed
        \end{proof}
    
    Utilising the stability condition, we now can provide an error estimate in one dimension.

    \begin{proposition}[Approximation Properties of $\tilde \Pi_{p,\Xi}$]\label{prop::missinglink}
        Let the assumptions of Proposition \ref{Prop::TildeStability} hold.
        For integers $1\leq s\leq p+1$ one finds
        \begin{align*}
            \norm{f-\tilde \Pi_{p,\Xi} f}_{L^2(I)}\lesssim h^{s}\norm{f}_{H^s(\tilde I)},\qquad \text{for all } f\in H^s(0,1),
            \intertext{and for integers $0\leq s\leq p$ one finds}
            \norm{f-\tilde \Pi_{p,\Xi}^\partial f}_{L^2(I)}\lesssim h^{s}\norm{f}_{H^s(\tilde I)},\qquad \text{for all } f\in H^s(0,1).
        \end{align*}
\end{proposition}
        \begin{proof} 
            We investigate merely the case of $\tilde\Pi_{p, \Xi}$. Due to the stability of $\tilde \Pi^\partial_{p,\Xi}$ as discussed in Proposition \ref{Prop::TildeStability}, we can prove the other case by similar means.

            For the first inequality, {as in \cite[Prop.~4.2]{Veiga_2014aa},} it is enough to consider classical polynomial estimates together with Proposition \ref{Prop::TildeStability} to achieve
            \begin{align*}
                \norm{f-\tilde\Pi_{p, \Xi} f}_{L^2(I)}&\leq \norm{f-q}_{L^2(I)}+\norm{\tilde\Pi_{p,\Xi}(q-f)}_{L^2(I)}\\
                &\lesssim \norm{f-q}_{L^2(I)}+ \norm{q-f}_{L^2(\tilde I)}+h\abs{q-f}_{H^1(\tilde I)} \\
                &\lesssim h^{s}\norm{f}_{H^s(\tilde I)},    
            \end{align*}
            which holds for a sensible choice of $q$, i.e., the $L^2$-orthogonal approximation w.r.t.~the polynomials of degree no higher than $p$.\qed
        \end{proof}

    We state the main result of this section.
    \begin{theorem}[Approximation via Commuting Multipatch Quasi-Interpolants]\label{lem::multiconv}
        Let Assumptions \ref{ass::knotvecs} and \ref{ass::geometry} be satisfied and let $s$ be integer-valued. Let 
                $f_0 \in H_{{\mathrm{pw}}}^{s}(\Gamma),\text{ }     2\leq s, $  as well as           
                $\bb f_1 \in   {\bb H}_{{\mathrm{pw}}}^{s}(\Gamma),\text{ } 1\leq s,$ and
                $f_2 \in  H_{{\mathrm{pw}}}^{s}(\Gamma),\text{ } 0\leq s.$
        Moreover, let each function be within the domain of the interpolation operator applied below, cf.~Lemma \ref{lem::multipatch::theycommute::and::regularity}.
        We find that
        \begin{align}
                \norm{f_0- \tilde \Pi_\Gamma^0f_0}_{L^2(\Gamma)} {}& {}\lesssim h^{s}\norm{f_0}_{H_{{\mathrm{pw}}}^s (\Gamma)},&  & 2\leq s\leq p+1, \notag\\
                \norm{f_0- \tilde \Pi_\Gamma^0f_0}_{H^1(\Gamma)} {}& {}\lesssim h^{s-1}\norm{f_0}_{H_{{\mathrm{pw}}}^{s} (\Gamma)},&  & 2\leq s\leq p+1, \notag\\
                \norm{\bb f_1- \tilde {\bb\Pi}_\Gamma^1 \bb f_1}_{\bb L^2(\Gamma)} {}  & {}\lesssim h^{s}\norm{\bb f_1}_{{\bb H}_{{\mathrm{pw}}}^{s}(\Gamma)},  &  & 1\leq s\leq p,   \notag \\
                \norm{f_2- \tilde\Pi_\Gamma^2 f_2}_{L^2(\Gamma)} {}                         & {}\lesssim h^{{s}}\norm{f_2}_{H_{{\mathrm{pw}}}^s (\Gamma)},    &  & 0\leq s\leq p. \notag
        \intertext{
        We moreover find that}
            \norm{\bb f_1- \tilde {\bb \Pi}_\Gamma^1 \bb f_1}_{\bb H^0(\div_\Gamma,\Gamma)} {}  & {}\lesssim h^s\norm{\bb f_1}_{{\bb H}_{{\mathrm{pw}}}^s(\div_\Gamma,\Gamma)},  &  & 1\leq s\leq p. \label{eq::hdivnorm}
        \end{align}
    \end{theorem}
        \begin{proof} 
            Due to the properties of the pull-backs and the locality of the norms involved, it suffices to provide a patchwise argument in the reference domain. 
            Note that the regularity of the spline approximation is always sufficient for the involved norms to be defined since it is enforced by the interpolation property of the $\tilde \Pi$ at the patch interfaces.

            \cite[Prop.~4.2]{Buffa_2015aa} directly provides
            \begin{align}
                \norm{f-\tilde\Pi_{\bb p,\bb \Xi}f}_{H^r\big((0,1)^2\big)}& \lesssim h^{s-r}\norm{f}_{H^{s}\big((0,1)^2\big)},
            \end{align}
            for $r=0,1$, from which the $\tilde \Pi_\Gamma^0$ case follows immediately.

            We will now provide a proof for the $\tilde\Pi_\Gamma^2$ case by investigating $\tilde\Pi^\partial_{\bb p,\bb \Xi}=\tilde\Pi^\partial_{p,\Xi_1}\otimes \tilde\Pi^\partial_{p,\Xi_2},$ which will be done largely analogous to the proofs within the cited literature. 
            The third assertion follows from a combination of the arguments in each vector component.

            Let $f\in H^s\big((0,1)^2\big)$ for some $1\leq s\leq p$. {Note that this implies $\norm{f}_{H^s_x(0,1)}\in{L_y^2(0,1)}$ and $\norm{f}_{H^s_y(0,1)}\in{L_x^2(0,1)}$, where by the $x$- and $y$-indexed norms we denote the usual norm taken w.r.t. the corresponding tensor product direction.}
            {Let $I_1\times I_2 = Q \subset (0,1)^2$ be an element.}
One can estimate via triangle inequality that
            \begin{align}
                \norm{f-\tilde \Pi^\partial_{\bb p,\bb \Xi} f}_{L^2(Q)} &=  \norm{f- (\tilde {\Pi}^\partial_{p,\Xi_1}\otimes  \tilde {\Pi}^\partial_{p,\Xi_2}) (f)}_{L^2(Q)}\notag\\
                & \leq \norm{f-(\tilde {\Pi}^\partial_{p,\Xi_1}\otimes  \operatorname{Id}) (f)}_{L^2(Q)} \notag\\ 
                &\qquad+\norm{(\tilde {\Pi}^\partial_{p,\Xi_1}\otimes  \operatorname{Id}) (f)  - (\tilde {\Pi}^\partial_{p,\Xi_1}\otimes  \tilde {\Pi}^\partial_{p,\Xi_2}) (f)}_{L^2(Q)}. \label{eq::estimate}
            \end{align}
            By Proposition~\ref{prop::missinglink} we immediately can estimate the first term of \eqref{eq::estimate} via
            \begin{align}
                \norm{f-(\tilde {\Pi}^\partial_{p,\Xi_1}\otimes  \operatorname{Id}) (f)}^2_{L^2(Q)}
                ={}&{}\int_{{I_2}} \norm{f-\tilde \Pi^\partial_{p,\Xi_1} f}^2_{L^2({I_1})}\opd y\notag\\
                \lesssim {}&{} h^{2s}\int_{{I_2}}\norm{f}_{H^s({\tilde I_1})}^2\opd y\notag\\
                \lesssim {}&{} h^{2s}\norm{f}_{H^s(\tilde Q)}^2.\label{proof::eqpack1}
            \end{align} 
            Now, we can estimate the second term of \eqref{eq::estimate} by utilisation of the stability property from Proposition \ref{Prop::TildeStability} {and application of Proposition~\ref{prop::missinglink}}, which yields
            \begin{align}
                \norm{(\tilde {\Pi}^\partial_{p,\Xi_1}\otimes  \operatorname{Id}) (f)  - (\tilde {\Pi}^\partial_{p,\Xi_1}\otimes  \tilde {\Pi}^\partial_{p,\Xi_2}) (f)}_{L^2(Q)}^2
                \lesssim{}&{}  \int_{{I_1}} \norm{f -\tilde\Pi^\partial_{p,\Xi_2}f }_{L^2({\tilde I_2})}^2\opd {x} \notag\\
                \lesssim {}&{} h^{2s}\norm{f}^2_{H^s(\tilde Q)}.\label{proof::eqpack2}
            \end{align}
            Now the assertion follows.
            Again, we stress that the missing assertion for an interpolator of type $\tilde\Pi\otimes\tilde\Pi^\partial$ follows analogously, even though it is not $L^2$-stable due to the impact of the seminorm term in \eqref{eq::notquiteL2stable}.
            One needs merely replace either \eqref{proof::eqpack1} or \eqref{proof::eqpack2} with the corresponding argument from \cite[{Prop.~4.2}]{Buffa_2015aa}.

               For an investigation of \eqref{eq::hdivnorm}, it suffices to utilise Lemma \ref{lem::multipatch::theycommute::and::regularity} together with the above to see that, for $1\leq s\leq p$, one finds
            \begin{align*}
                \norm{\bb f_1- \tilde {\bb\Pi}_{\bb p,\bb\Xi}^1 \bb f_1}_{\bb H^0(\div,(0,1)^2)} {}  
                & {}\leq \norm{\bb f_1- \tilde {\bb\Pi}_{\bb p,\bb\Xi}^1 \bb f_1}_{\bb L^2((0,1)^2)} + \norm{\div(\bb f_1- \tilde {\bb\Pi}_{\bb p,\bb\Xi}^1 \bb f_1)}_{L^2((0,1)^2)}\\
                &{}={} \norm{\bb f_1- \tilde {\bb\Pi}_{\bb p,\bb\Xi}^1 \bb f_1}_{\bb L^2((0,1)^2)} + \norm{\div\bb f_1-  \div(\tilde {\bb\Pi}_{\bb p,\bb\Xi}^1 \bb f_1)}_{L^2((0,1)^2)}\\
                &{}={} \norm{\bb f_1- \tilde {\bb\Pi}_{\bb p,\bb\Xi}^1 \bb f_1}_{\bb L^2((0,1)^2)} + \norm{\div\bb f_1-  \tilde{\bb\Pi}_{\bb p,\bb\Xi}^2 \div ( \bb f_1)}_{L^2((0,1)^2)}\\
                & {} \lesssim h^{s}\norm{\bb f_1}_{{\bb H}^s(\div,(0,1)^2)},
            \end{align*}
            from which the result follows by properties of the geometry mapping.\qed
        \end{proof}
   
        These results are immediately applicable to two-dimensional finite element methods with a straight forward generalisation to three dimensions, see \appendA.

        \begin{corollary}[Approximation Results for Finite Element Methods]\label{cor::femanalogy}
        	Let $\Omega$ be a two dimensional domain, satisfying Assumption \ref{ass::geometry}.
        	Let $f_0\in H^1(\Omega)$, $\bb f_1\in \bb H^0(\div,\Omega)$ and $f_2\in L^2(\Omega)$. Then, its holds that
        	\begin{align*}	
        	 \inf_{f_h\in \S^0_{\bb p,\bb \Xi}(\Omega)}\norm{ f_0 - f_h}_{H^{1}(\Omega)}&\lesssim h^{s-1}\norm{f_0}_{{H}_{{\mathrm{pw}}}^{s}(\Omega)},&1\leq {}&{}s \leq p+1,\\
       		 \inf_{\bb f_h\in \bb \S^1_{\bb p,\bb \Xi}(\Omega)}\norm{ \bb f_1 - \bb f_h}_{\bb H^0(\div,\Omega)}&\lesssim  h^{s}\norm{\bb f_1}_{  {\bb H}_{{\mathrm{pw}}}^s(\div,\Omega_j)},&0\leq {}&{}s \leq p,\\
        	 \inf_{f_h\in \S^2_{\bb p,\bb \Xi}(\Omega)}\norm{ f_2 - f_h}_{L^2(\Omega)}&\lesssim h^{s}\norm{f_2}_{{H}_{{\mathrm{pw}}}^{s}(\Omega)},&0\leq {}&{}s \leq p.
        	\end{align*}
        \end{corollary}
        	\begin{proof}
        		Due to the stability of the respective orthogonal projection $\mathcal P_1\colon H^{1}(\Omega)\to \S^0_{\bb p,\bb \Xi}(\Omega)$, 
        		${\mathcal P}_{\div}\colon {\bb{ H}^0}(\div,\Omega)\to {\bb \S}^1_{\bb p,\bb \Xi}(\Omega)$ 
        		and $\mathcal P_0\colon L^2(\Omega) \to \S^2_{\bb p,\bb \Xi}(\Omega)$, we immediately have the result for the minimal values of $s$.
        		{Repeating the same steps as in the proof of} Theorem \ref{lem::multiconv}, we find the result for larger values of $s$ and smooth choices of $f_0$, $\bb f_1$ and $f_2$. The assertion now follows by interpolation arguments {as in Lemma \ref{interpolationlemma}} and {density of smooth functions in (subspaces of) $L^2(\Omega)$}.
            \qed
        	\end{proof}

        Again, a generalisation of this result to the sequence
        $$ \begin{tikzcd}
            H^{1}(\Omega)\ar{r}{\bb\grad}&\bb H^0(\bb\curl,{\Omega})\ar{r}{\bb\curl}&\bb H^0(\div,\Omega)\ar{r}{\div}&L^2(\Omega)
        \end{tikzcd}
		$$
        on three-dimensional volumetric domains $\Omega$ is straight forward, cf. Appendix A. This includes in particular also the approximation property of $\bb H^0(\bb \curl,\Omega)$, which, for two-dimensional domains, coincides with the $\bb H^0(\div,\Omega)$-case, up to rotation, see \cite[Sec.~5.5]{Veiga_2014aa}.
        
        \begin{remark}
        {We emphasise that the interpolation operators constructed in this section are merely a theoretical tool for which there are alternatives, cf.~\cite{Gerritsma2010} or the sources cited therein.
        We utilise the Schumaker quasi-interpolation operators \cite{Schumaker_2007aa}, since they are often used within the spline community and they suffice to show quasi-optimality w.r.t. $h$.
        In most implementations, it suffices to implement the orthogonal projection or evaluate a suitable bilinear form via quadrature rules rather than interpolation operators.}
        \end{remark}

    \section{Approximation Properties in Trace Spaces}\label{sec::tracespaces}

    Now, we will consider approximation properties of the spaces $\S^0_{\bb p,\bb \Xi}(\Gamma)$, $\bb\S^1_{\bb p,\bb \Xi}(\Gamma)$ and $\S^2_{\bb p,\bb \Xi}(\Gamma)$ w.r.t.~the fractional Sobolev spaces $H^{1/2}(\Gamma)$,  $\bb H_\times^{-1/2}(\div_\Gamma,\Gamma)$  and $H^{-1/2}(\Gamma)$.
	
	This will be done by investigation of the orthogonal projection. Due to its optimality, we know that it must achieve the same convergence rates w.r.t.~$h$-refinement as those of Theorem~\ref{lem::multiconv}.
    Moreover, properties of the orthogonal projection are of interest for an application in the context of partial differential equations, since inf-sup conditions yield quasi-optimal behaviour for the approximate solution w.r.t.~the orthogonal approximation of the involved energy space \cite{Xu_2002aa}.

    We will start by utilisation of interpolation as in Lemma \ref{interpolationlemma} and optimality of the orthogonal projection of the respective energy space to get convergence results for positive fractional spaces. This yields the following corollary.
\begin{corollary}[Approximating $H^{1/2}(\Gamma)$ with $\S_{\bb p,\bb \Xi}^0(\Gamma)$]\label{H1/2}
    Let Assumptions \ref{ass::knotvecs} and \ref{ass::geometry} be satisfied.
    Let {$f\in H^{s}_{{\mathrm{pw}}}(\Gamma)\cap H^{1/2}(\Gamma)$} 
{for integers $2\leq s\leq p+1$}, and let $\mathcal P_{1/2} f$ denote its $H^{1/2}(\Gamma)$-orthogonal projection onto $\S_{\bb p,\bb \Xi}^0(\Gamma).$ It holds that
    \begin{align*}
    \norm{f- \mathcal P_{1/2}f}_{H^{1/2}(\Gamma)} {}& {}\ \lesssim h^{s-1/2}\norm{f}_{H^{s}_{{\mathrm{pw}}}(\Gamma)}.
    \end{align*}
\end{corollary}
    \begin{proof} 
        By Theorem \ref{lem::multiconv} we know for integers $s$ with $2 \le s \le p+1$ that
        \begin{align*}
        	\norm{f- \tilde \Pi^0_\Gamma(f)}_{H^{r}(\Gamma)} &{}\lesssim h^{s-r}\norm{f}_{{H}^s_{{\mathrm{pw}}}(\Gamma)},
        \end{align*}
for both $r \in \{0,1\}$. Now, application of Lemma \ref{interpolationlemma} yields
        \begin{align*}
        	\norm{f- \tilde \Pi^0_\Gamma(f)}_{H^{1/2}(\Gamma)} &{}\lesssim h^{s-1/2}\norm{f}_{{H}^s_{{\mathrm{pw}}}(\Gamma)}.
        \end{align*}
By optimality of the $H^{1/2}(\Gamma)$-orthogonal projection $\mathcal P_{1/2}$, we {obtain the result.} \qed
    \end{proof}

Interpolation does not yield estimates in norms with negative index. Thus, to show the approximation properties of $\S_{\bb p,\bb \Xi}^2(\Gamma)$ in $H^{-1/2}(\Gamma),$ we resort to an application of the Aubin-Nitsche Lemma \cite{Adams_1978aa}.
\begin{corollary}[Approximating $H^{-1/2}(\Gamma)$ with $\S_{\bb p,\bb \Xi}^2(\Gamma)$]\label{H-1/2}
    Let Assumptions \ref{ass::knotvecs} and \ref{ass::geometry} be satisfied.
    Let $f\in H^{-1/2}(\Gamma)\cap H_{{\mathrm{pw}}}^{s}(\Gamma)$ for some $s\geq {0}$. Let $\mathcal P_{-1/2}$ denote the $H^{-1/2}(\Gamma)$-orthogonal projection of $f$ onto $\S_{\bb p,\bb \Xi}^2(\Gamma).$
    Then it holds that
    \begin{align}
        \norm{f- \mathcal P_{-1/2}f}_{H^{-1/2}(\Gamma)} {} & {} \lesssim h^{{s+1/2}}\norm{f}_{H_{{\mathrm{pw}}}^s (\Gamma)},    &  & {0}\leq s\leq p. \label{eq::h-12}
    \end{align}
\end{corollary}
    \begin{proof} 
    Assume, for now, that $f\in L^2(\Gamma)\cap H_{{\mathrm{pw}}}^{s}(\Gamma),$ and let $\mathcal P_0$ denote the $L^2$-orthogonal approximation onto $\S_{\bb p,\bb \Xi}^2(\Gamma).$
    Since $H^{-1/2}(\Gamma)$ is the dual space to $H^{1/2}(\Gamma)$ we can estimate
        \begin{align}\begin{aligned}
                    \norm{f- \mathcal P_0 f}_{H^{-1/2}(\Gamma)}\coloneqq {}&{}\sup_{0\neq v\in H^{1/2}(\Gamma)} \frac{\abs{\langle f- \mathcal P_0 f , v\rangle_{L^2(\Gamma)}}}{\norm{v}_{H^{1/2}(\Gamma)}}\\ =
                    {}&{}\sup_{0\neq v\in H^{1/2}(\Gamma)} \frac{\abs{\langle f- \mathcal P_0 f , v-\mathcal P_0 v\rangle_{L^2(\Gamma)}}}{\norm{v}_{H^{1/2}(\Gamma)}}\\
                    \lesssim {}&{} \norm{f-\mathcal P_0 f}_{L^2(\Gamma)}\sup_{0\neq v\in H^{1/2}(\Gamma)}\frac{\norm{v-\mathcal P_0 v}_{L^2(\Gamma)}}{\norm{v}_{H^{1/2}(\Gamma)}}
                    .\end{aligned}\label{classicaldualityhalforderconvergence}
        \end{align}
        By Theorem \ref{lem::multiconv}, we now arrive at 
        $\norm{f-\mathcal P_{0}f }_{H^{-1/2}(\Gamma)}\leq h^{1/2+s}\norm{f}_{H_{{\mathrm{pw}}}^s(\Gamma)}$ for $0\leq s\leq p.$ Replacing $\mathcal P_0$ by $\mathcal P_{-1/2}$ now yields the assertion, analogously to the proof of Corollary \ref{H1/2}, using interpolation, optimality of $\mathcal P_{-1/2}$ w.r.t.~the $H^{-1/2}(\Gamma)$-error and density of regular functions in $H^{-1/2}(\Gamma)$.\qed
    \end{proof}

\begin{remark}
	This result does not necessarily rely on Theorem \ref{lem::multiconv}. Since $H^{-1/2}(\Gamma)$ allows for discontinuities, it can be reproduced by application of the patchwise estimates of Corollary \ref{cor::Approxcor}. This has been done in \cite{Doelz_2017aa}.
\end{remark}

\begin{remark}
	{Note that by putting global norms on the right hand side, analogues of Corollaries \ref{H1/2} and \ref{H-1/2} can be shown for minimal regularities, i.e., $1/2\leq s$ in the case of the $H^{1/2}(\Gamma)$-error and $-1/2\leq s$ in the sense of the $H^{-1/2}(\Gamma)$-error by almost analogous means, cf.~\cite[Thm.~4.2.17]{Sauter_2010aa}. However, these results rely on the smoothness of the geometry for the norm on the right hand side to be well defined. 
  We aim for our results to be immediately applicable to the multipatch setting of isogeometric analysis, where we want to require smoothness of the geometry only patchwise.}
\end{remark}

Now, what is missing to understand the approximation properties of the spaces $\S^0_{\bb p,\bb\Xi}(\Gamma),$ $\bb\S_{\bb p,\bb\Xi}^1(\Gamma),$ and $\S^2_{\bb p,\bb\Xi}(\Gamma)$ in the trace space setting w.r.t.~the diagram in Figure \ref{fig::classicaldeRham}, is an analysis of the approximation properties of $\bb \S_{\bb p,\bb\Xi}^1(\Gamma)$ in the space $\bb H_\times^{-1/2}(\div_\Gamma,\Gamma).$ 

For this purpose, we want to employ an argument similar to the one of Corollary \ref{H-1/2}. 
However, as will be discussed in a moment, this cannot be done with such ease as before in Corollary \ref{H-1/2}. We choose to follow the lines of \cite{Buffa_2003ab}, from whose argumentation we deviate only to adapt to the B-spline setting. The proof is lengthy and technical, thus we only state the result, with the full proof discussed in Section \ref{longproof}.

\begin{theorem}[Approximating $\bb H_\times^{-1/2}(\div_\Gamma,\Gamma)$ with $\bb \S_{\bb p,\bb\Xi}^1(\Gamma)$]\label{thm::hdiv}
    Let Assumptions \ref{ass::knotvecs} and \ref{ass::geometry} be satisfied and
    let $\bb f\in {\bb H}_{{{\mathrm{pw}}}}^s(\div_\Gamma,\Gamma)\cap \bb H_\times^{-1/2}(\div_\Gamma,\Gamma)$ for some $s\geq -1/2.$ Let $\mathcal P_\times f$ denote the $\bb H_\times^{-1/2}(\div_\Gamma,\Gamma)$-orthogonal projection of $\bb f$ onto $\bb\S^1_{\bb p,\bb \Xi}(\Gamma)$.
    Then one finds 
    \begin{align*}
        \norm{\bb f-\mathcal P_\times f}_{\bb H_\times^{-1/2}(\div_\Gamma,\Gamma)} &\lesssim h^{1/2+s}  \norm{\bb f}_{{\bb H}_\times^s(\div_\Gamma,\Gamma)},
\intertext{ for all $-1/2\leq s\leq 0.$ Moreover, for $0\leq s\leq p$, it holds that}
        \norm{\bb f-\mathcal P_\times f}_{\bb H_\times^{-1/2}(\div_\Gamma,\Gamma)}& \lesssim h^{1/2+s}  \norm{\bb f}_{  {\bb H}_{{\mathrm{pw}}}^s(\div_\Gamma,\Gamma)}. 
    \end{align*}
\end{theorem}

\begin{remark}
    Note that Corollary \ref{H-1/2} and Theorem \ref{thm::hdiv} include the classical results from boundary element theory, even though a first glance suggests otherwise. This is due to the fact, that $p$ refers not to the degree of $\S_{\bb p,\bb \Xi}^2(\Gamma)$ and $\bb \S_{\bb p,\bb \Xi}^1(\Gamma)$ respectively, but rather to the degree at the beginning of the sequence $\S_{\bb p,\bb \Xi}^0(\Gamma)\to\bb \S_{\bb p,\bb \Xi}^{1}(\Gamma)\to\S_{\bb p,\bb \Xi}^2(\Gamma).$ In terms of basis functions, the space $\S_{\bb p,\bb \Xi}^2(\Gamma)$ contains splines of degree $p-1$, thus shifting the notation by 1.
\end{remark}

\subsection{Proof of Theorem 3}\label{longproof}

Within this section, we provide a detailed proof of Theorem \ref{thm::hdiv}, by means of a patch by patch duality argument, similar to the one utilised to achive the estimate in $H^{-1/2}(\Gamma)$.
However, one problem with a na\"ive patchwise argument is due to the fact, that $\bb H_\times^{-1/2}(\div_\Gamma,\Gamma)$ incorporates a (weak) continuity across the patch normals w.r.t.~$\partial\Gamma_j$, see \cite{Buffa_2003aa}. 
Thus an orthogonal approximation required for an Aubin-Nitsche type argument cannot easily be localised to a single patch.

This problem can be overcome by defining a suitable projection manually, by applying the orthogonal projection only on the part without outgoing flux and localising the approximation of the outgoing fluxes to the patch boundaries.
For this purpose, we define the space $K_j^s$ as the kernel of the local trace operator 
$$\gamma_{\bb n,j}(\bb f)(\bb x_0)\coloneqq\lim_{\Gamma_j\ni \tilde {\bb x} \to \bb x_0} \bb f(\tilde {\bb x})\cdot \bb n_{\bb x_0},\qquad\text{for all } \bb x_0\in \partial\Gamma_j,$$
on $\bb H^s(\div_\Gamma,\Gamma_j)$, where $\bb n_{\bb x_0}$ denotes the outer unit normal w.r.t.~$\partial \Gamma_j$ at $\bb x_0\in \partial\Gamma_j.$ 
The same way we denote the kernel of $\gamma_{\bb n,j}$ on $\bb\S_{\bb p,\bb\Xi}^1(\Gamma_j)$ by $K^{\mathbb S}_{j}.$
{Note that $\gamma_{\bb n}\colon \bb H^0(\div,(0,1)^2)\to H^{-1/2}(\partial (0,1)^2)$ is continuous \cite[Thm.~2.5]{Girault1986}. Due to Assumption \ref{ass::geometry} this immediately transfers to $$\gamma_{\bb n,j}\colon \bb H^0(\div_\Gamma,\Gamma_j)\to H^{-1/2}(\partial \Gamma_j).$$}
{Here, $H^{-1/2}(\partial \Gamma_j)$ has to be understood as a mapped counterpart to $H^{-1/2}(\partial (0,1)^2)$, in complete analogy to the definition for two-dimensional domains, cf.~\cite[p.~96ff]{McLean_2000aa}. This definition is, again, valid due to Assumption \ref{ass::geometry}.}

{\begin{remark}[Local Shifting Property]\label{rem::localshiftinggamman}
    We remark that $\gamma_{\bb n,j}$ enjoys a local shifting property, in the sense that $\gamma_{\bb n,j} \colon \bb H^s(\Gamma_j)\to H^{s-1/2}(\Gamma_{j,i})$ is continuous for $s>1/2$, which implies continuity of $\gamma_{\bb n,j} \colon \bb H^s(\div_\Gamma,\Gamma_j)\to H^{s-1/2}(\partial\Gamma_{j,i})$. Here, $\partial\Gamma_{j,i}$ denotes one of the four ``sides'' of $\partial\Gamma_j$. 
    This can be seen, since in the reference domain, and restricted to one side of $\partial (0,1)^2$ the identity $\gamma_{\bb n}(\bb f) = \gamma_0 ( \bb e_i^\intercal \cdot \bb f)$ holds, where $\bb e_i$ is either $(1,0)^\intercal$ or $(0,1)^\intercal$, depending on the side.
    This allows us to locally utilise the canonical continuity assertions of $\gamma_0$, cf.~\cite[Thm.~3.37]{McLean_2000aa}.
\end{remark}}

We {now proceed} by reviewing two technical results.

\begin{lemma}[Continuity Estimate, {\cite[Lem.~4.8]{Buffa_2003ab}}]\label{LemmaToBeDone}
    Fix a patch $\Gamma_j$. Let $\zeta\in H^{-1/2}(\partial\Gamma_j)$ such that $\langle\zeta,1\rangle_{{H^{-1/2}(\partial\Gamma_j)\times H^{1/2}(\partial\Gamma_j)}}=0$ holds. Let $\xi\in\bb H^0(\div_\Gamma,\Gamma_j)$ be the solution to the problem
        $\langle \xi,v\rangle_{\bb H^0(\div_\Gamma,\Gamma_j)} = 0$
    for all $v\in K^0_j$ with $\gamma_{\bb n,j}(\xi)= \zeta$ onto $\partial \Gamma_j$.
    Then one finds
         $\norm{\xi}_{\bb H^{s+1/2}(\div_\Gamma,\Gamma_j)}\leq C_s\norm{\zeta}_{H^s(\partial\Gamma_j)}$
    for both $s=-1$ and $s=-1/2$.
\end{lemma}

\begin{lemma}[Duality Relation, {\cite[Lem.~4.7]{Buffa_2003ab}}]\label{lem::isomorph}
Let $(K^{-1/2}_j)'$ denote the dual of $K^{-1/2}_j$ w.r.t. $\bb H^0(\div_\Gamma,\Gamma_j)$. 
There exists {an} isomorphism $ K_j^{1/2} \to (K^{-1/2}_j)'$.
\end{lemma}

{
In the following we will be interested in the spaces $\bb H_\mathrm{pw}^{s}(\div_\Gamma,\Gamma)$, which we equip with the norm $\sum_{0\leq j<N}\Vert\cdot\Vert_{\bb H^s_\mathrm{pw}(\div_\Gamma,\Gamma_j)}$ for $s\geq -1/2.$
Note that $\bb H_\mathrm{pw}^{-1/2}(\div_\Gamma,\Gamma)$ is continuously embeddable in $\bb H_\times^{-1/2}(\div_\Gamma,\Gamma),$ 
cf.~the first inequality of \cite[Eq.~(98)]{Buffa_2003ab}.}


\begin{definition}[Conforming Projection]\label{def::helpingprojection}
	For {$\bb g\in \bb H^1_{\mathrm{pw}}(\Gamma)$} we define the projection $\pi$ onto $\bb \S_{\bb p,\bb \Xi}^1(\Gamma)$ as the solution to the problem
        \begin{align}
                \pi_j \bb g\coloneqq {}&{} (\pi \bb g)|_{\Gamma_j} ,\notag\\
                \langle \pi_j \bb g - \bb g|_{\Gamma_j},\bb b \rangle_{\bb  H^0(\div_\Gamma,\Gamma_j)} = {} &{}0,\qquad \forall \bb b \in K^{\mathbb S}_j,\ \forall 0\leq j< N,\label{proof::eq::kernelorthogonal}\\
                \langle \gamma_{\bb n,j}(\pi_j\bb g) - \gamma_{\bb n,j}(\bb g),\gamma_{\bb n,j}(\bb b)\rangle_{L^2(\partial\Gamma_j)}  = {}&{}0  ,\qquad \forall \bb b \in N^{\mathbb S}_j,\ \forall 0\leq j< N.\label{proof::eq::outgoingfluxL2}
        \end{align}
       {Herein, we use the decomposition         
       \begin{align}
          \bb\S_{\bb p,\bb \Xi}^1({\Gamma_j})= N^{\mathbb S}_j\oplus K^{\mathbb S}_j,\label{compo20}
        \end{align} 
        {where $N^{\mathbb S}$ denotes the span of basis functions with non-vanishing normal trace $\gamma_{\bb n,j}$.}
        {Note that this induces a unique decomposition of every function in $ \bb b \in\bb \S^1_{\bb p,\bb\Xi}(\Gamma)$, since it is clear that both  $N^{\mathbb S}_j$ and $K^{\mathbb S}_j$ can be identified with specific, disjoint sets of degrees of freedom, i.e., are discrete and closed subspaces of $\bb \S^1_{\bb p,\bb\Xi}(\Gamma)$.}
}
\end{definition}

		The idea behind this projection is similar to projections in the context of mixed finite element methods, which is equal to the face by face projection that preserves boundary data on interfaces, see \cite{Boffi2013}.
	 	It chooses the part without outgoing flux as the optimal approximation w.r.t.~the $\bb H^0(\div_\Gamma,\Gamma_j)$-norm, and the part incorporating outgoing fluxes as optimal w.r.t.~the $L^2(\partial\Gamma_j)$-norm.
        Since the outgoing flux is continuous across patch boundaries, \eqref{proof::eq::outgoingfluxL2} ensures the same for the discretisation.

        {Note that the projection is indeed well defined with respect to the composition \eqref{compo20} since each of the lines \eqref{proof::eq::kernelorthogonal} and \eqref{proof::eq::outgoingfluxL2} fixes a unique element of $N^{\mathbb S}_j$ or $K^{\mathbb S}_j$, respectively.}

{\begin{remark}[Locality of the $L^2(\partial\Gamma_j)$-Projection]\label{rem::loc::l2partialgammaj}
    We remark that, due to the structure of the spline space {and locality of the $L^2(\partial\Gamma_j)$-scalar product}, the $L^2(\partial\Gamma_j)$-orthogonal projection in \eqref{proof::eq::outgoingfluxL2} is equivalent to application of the projection to each side $\partial\Gamma_{j,i}$ of $\partial \Gamma_j$ separately.
\end{remark}}

To argue that the projector $\pi_j$ has the expected approximation properties, we require a discrete right-inverse of $\gamma_{\bb n,j}.$ Such constructions are readily available, often via approximation of the continuous right inverse. This requires some boundary value preserving interpolation to be $\bb H(\div,\Gamma)$-stable, which is not satisfied by $\tilde{\bb \Pi}^1_{\Gamma}.$ However, a similar result is satisfied. By application of an inverse estimate for polynomials to \eqref{eq::notquiteL2stable} one observes that the one-dimensional interpolant $\tilde \Pi$ is $L^2(0,1)$-stable for piecewise polynomial inputs. The following construction builds on this observation.

\begin{proposition}[Discrete Right Inverse of the Normal Trace]\label{prop::disc::right::inverse}
    There is a discrete right inverse $$\bb R_{h,j}\colon \gamma_{\bb n,j}(\bb \S^1_{\bb p,\bb \Xi}(\Gamma_j))\to \bb \S^1_{\bb p,\bb \Xi}(\Gamma_j)$$ to $\gamma_{\bb n,j}$ 
    which is continuous in the sense of $H^{-1/2}(\partial \Gamma_j)\to \bb H^0(\div_\Gamma,\Gamma_j)$.
\end{proposition}
    \begin{proof}
        The proof is simple yet technical. W.l.o.g. we conduct the argument in the reference domain. 
        First, we note that there exists a Raviart-Thomas space, that we denote by ${\mathbb{Q}}_{\bb p}$, 
consisting of elements $\bb Q_{p}= Q_{p,p-1}\times Q_{p-1,p},$ cf.~\cite{Costabel_polynomialextension} on $[0,1]^2$ such that $\bb \S^1_{\bb p,\bb \Xi}([0,1]^2)\subseteq {\mathbb{Q}}_{\bb p}.$ 

        For ${\mathbb{Q}}_{\bb p}$, the existence of a continuous right inverse $\bb R_{{\mathbb{Q}}_{\bb p}}$ 
is known, cf.~\cite[Thm.~4.1.9]{Quarteroni_1999aa} for lowest-order Raviart-Thomas elements. The same construction can be applied straightforwardly for high order elements, as it relies on the existence of stable quasi-interpolation operators \cite[Eq.~(2.5.26)]{Boffi2013}, and the increased regularity of the continuous (as opposite to discrete) right inverse, see again \cite[Thm.~4.1.9]{Quarteroni_1999aa} for details. See also \cite[Thm. 3.10]{Costabel_polynomialextension}, noting that in two dimensions the curl conforming spaces correspond to a rotation of the divergence conforming ones.
        
        As a second step, we set our lifting as $\bb R_{h,j} = \tilde{\bb \Pi}^1_{\bb p,\bb \Xi} \circ \bb R_{{\mathbb{Q}}_{\bb p}}$
and show continuity. For this, we note that the $\tilde \Pi$ operators commute with the surface differential operators, and that by Proposition \ref{Prop::TildeStability} and the tensor-product construction \eqref{def::commtilde} the operator $\tilde{\Pi}^2_{\bb p,\bb \Xi}$ is $L^2$-stable. For $u\in \gamma_{\bb n}(\bb \S^1_{\bb p,\bb \Xi}(\Gamma_j))\subseteq \gamma_{\bb n}({\mathbb{Q}}_{\bb p})$
        we estimate
        \begin{align*}
            &\norm{\tilde{\bb \Pi}^1_{\bb p,\bb \Xi} (\bb R_{{\mathbb{Q}}_{\bb p}} u)}^2_{\bb H^0(\div,(0,1)^2)} \\
             ={}&{} 
            \norm{\tilde{\bb \Pi}^1_{\bb p,\bb \Xi} (\bb R_{{\mathbb{Q}}_{\bb p}} u)}^2_{\bb L^2((0,1)^2)} + \norm{ \div (\tilde{\bb \Pi}^1_{\bb p,\bb \Xi} (\bb R_{{\mathbb{Q}}_{\bb p}} u)) }^2_{L^2((0,1)^2)} \\
            = {}&{} \norm{\tilde{\bb \Pi}^1_{\bb p,\bb \Xi} (\bb R_{{\mathbb{Q}}_{\bb p}} u)}^2_{\bb L^2((0,1)^2)} + \norm{\tilde{\Pi}^2_{\bb p,\bb \Xi} (\div  (\bb R_{{\mathbb{Q}}_{\bb p}} u))}^2_{L^2((0,1)^2)}\\
             \lesssim {}&{}\norm{\tilde{\bb \Pi}^1_{\bb p,\bb \Xi} (\bb R_{{\mathbb{Q}}_{\bb p}} u)}^2_{\bb L^2((0,1)^2)} + \norm{\div  (\bb R_{{\mathbb{Q}}_{\bb p}} u)}^2_{ L^2((0,1)^2)}.
        \end{align*}

        For the first term we estimate only one vector component, since the estimate for the second vector component follows analogously. Let $R_x u$ denote the first vector component of $\bb R_{{\mathbb{Q}}_{\bb p}} u$. Making the tensor product structure explicit, we apply the assertions of Proposition \ref{Prop::TildeStability} which yields
        \begin{align*}
            &\norm{(\tilde  \Pi^\partial_{p,\Xi} \otimes \tilde \Pi_{p,\Xi}) R_x u}_{L^2((0,1)^2)}^2 
             \lesssim 
            \norm{(\operatorname{Id} \otimes \tilde  \Pi_{p,\Xi}) R_x u}_{L^2((0,1)^2)}^2 \\
            \leq {}&{}\int_0^1 \norm{(\tilde  \Pi_{p,\Xi} \circ R_x) (u(x,\cdot))}_{L^2(0,1)}^2 \opd x\\
            \lesssim{}&{}\int_0^1 \left(\norm{ (R_x u)(x,\cdot)}_{L^2(0,1)} + h\abs{( R_x u)(x,\cdot)}_{H^1(0,1)}\right)^2\opd x.
        \end{align*}
        Note that, by  choice of $\bb \S^1_{\bb p,\bb \Xi}([0,1]^2)\subseteq {\mathbb{Q}}_{\bb p}$ the seminorm term is well-defined. Since $R_x u(x,\cdot)$ is a continuous piecewise polynomial for any $x\in[0,1]$, we can apply inverse estimates \cite[Lem.~4.5.3]{BrennerScott}, i.e., $h\abs{ (R_x u)(x,\cdot)}_{H^1(0,1)} \lesssim \norm{(R_x u)(x,\cdot)}_{L^2(0,1)}.$ This together with the above yields 
        \begin{align*}
                \norm{\tilde{\bb \Pi}^1_{\bb p,\bb \Xi} (\bb R_{{\mathbb{Q}}_{\bb p}} u)}_{\bb H^0(\div,(0,1)^2)}\lesssim \norm{\bb R_{{\mathbb{Q}}_{\bb p}} u}_{\bb H^0(\div,(0,1)^2)}.
        \end{align*} 
        As a third step, we invoke the continuity of $\bb R_{{\mathbb{Q}}_{\bb p}}$ and pull-back to the physical domain. The assertion follows.
        \qed
    \end{proof}
\color{black}

\begin{lemma}[Convergence Property]\label{lem::quasioptimal}
    For any $0\leq j<N$ and 
    $\bb u{} {{}\in {\bb H}^s(\div_\Gamma,\Gamma_j)}$
    the projection $\pi_j$ defined in Definition \ref{def::helpingprojection} fulfills 
    \begin{align*}
        \norm{\bb u -\pi_j\bb u}_{\bb H^0(\div_\Gamma,\Gamma_j)}\lesssim h^s\norm{\bb u}_{{\bb H}^s(\div_\Gamma,\Gamma_j)},\qquad {1\leq} s\leq p.
    \end{align*}
\end{lemma}
    \begin{proof}
{
Let us define the subspace of discrete functions whose normal trace coincides with $\pi_j\bb u$, i.e., the space
\[
\bb W = \{\bb w \in \bb\S_{\bb p,\bb \Xi}^1(\Gamma_j) : \gamma_{\bb n,j}( \bb w ) = \gamma_{\bb n,j}(\pi_j \bb u )\} = \{ \bb w \in \bb\S_{\bb p,\bb \Xi}^1(\Gamma_j) : \bb w - \pi_j \bb u \in K^\S_j\}.
\]
In complete analogy {to the proof of the Céa Lemma, see \cite[Eq. (2.8.1)]{BrennerScott},} we know that
\begin{equation} \label{eq:infw}
\norm{\bb u -\pi_j\bb u}_{\bb H^0(\div_\Gamma,\Gamma_j)} \lesssim \inf_{\bb w \in \bb W} \norm{\bb u - \bb w}_{\bb H^0(\div_\Gamma,\Gamma_j)}.
\end{equation}
Let us define 
\[
\bb w = \bb \Pi_{\bb p,\bb \Xi}^1\bb u - \bb R_{h,j} \left(\gamma_{\bb n,j}( \bb \Pi_{\bb p,\bb \Xi}^1\bb u) - \gamma_{\bb n,j}(\pi_j \bb u )\right),
\]
where $\bb \Pi_{\bb p,\bb \Xi}^1$ is the single patch operator {of} Lemma~\ref{lem::theycommute} lifted to $\Gamma_j$. By definition, it is immediate to see that $\gamma_{\bb n,j}( \bb w ) = \gamma_{\bb n,j}(\pi_j \bb u )$, and therefore one finds that $\bb w \in \bb W$.
The triangle inequality yields 
\begin{align*}
\norm{\bb u - \bb w}_{\bb H^0(\div_\Gamma,\Gamma_j)} \le 
\norm{\bb u - \bb \Pi_{\bb p,\bb \Xi}^1\bb u}_{\bb H^0(\div_\Gamma,\Gamma_j)} + 
\norm{\bb R_{h,j} \left(\gamma_{\bb n,j}( \bb \Pi_{\bb p,\bb \Xi}^1\bb u) - \gamma_{\bb n,j}(\pi_j \bb u )\right)}_{\bb H^0(\div_\Gamma,\Gamma_j)}.
\end{align*}
The first term on the right can be estimated by
\begin{align*}
    \norm{\bb u - \bb \Pi_{\bb p,\bb \Xi}^1\bb u}_{\bb H^0(\div_\Gamma,\Gamma_j)} \lesssim h^s \norm{\bb u}_{\bb H^s(\div_\Gamma,\Gamma_j)},
\end{align*}
due to Corollary~\ref{cor::Approxcor}.
The second term can be bounded from above by
\begin{align*}
    {}&{}\norm{\bb R_{h,j} \left(\gamma_{\bb n,j}( \bb \Pi_{\bb p,\bb \Xi}^1\bb u) - \gamma_{\bb n,j}(\pi_j \bb u )\right)}_{\bb H^0(\div_\Gamma,\Gamma_j)}\notag\\
\lesssim {}&{}\norm{\gamma_{\bb n,j}( \bb \Pi_{\bb p,\bb \Xi}^1\bb u) - \gamma_{\bb n,j}(\pi_j \bb u ) + \gamma_{\bb n,j}(\bb u ) - \gamma_{\bb n,j}(\bb u )}_{\bb H^{-1/2}(\partial\Gamma_j)} \\
\lesssim{}&{} {\norm{\gamma_{\bb n,j}(\bb \Pi_{\bb p,\bb \Xi}^1\bb u - \bb u)}_{\bb H^{-1/2}(\partial\Gamma_j)}}+ \norm{\gamma_{\bb n,j}(\bb u - \pi_j \bb u) }_{\bb H^{-1/2}(\partial\Gamma_j)}\\
{\lesssim{}}{}&{} {\norm{\bb \Pi_{\bb p,\bb \Xi}^1\bb u - \bb u}_{\bb H^0(\div_\Gamma,\Gamma_j)} + \norm{ \gamma_{\bb n,j}(\bb u) - (\mathcal P_0\circ\gamma_{\bb n,j})( \bb u )}_{\bb H^{-1/2}(\partial\Gamma_j)}.}
\end{align*}
{Here, we use the continuity of the lifting $\bb R_{h,j},$ followed by a triangle inequality, and the continuity of $\gamma_{\bb n,j}$ together with the identity  $\mathcal P_0\circ \gamma_{\bb n,j} = \gamma_{\bb n,j}\circ \pi_j$.}
{The first term can be estimated by application of Corollary~\ref{cor::Approxcor} while the second term can be handled via duality arguments. One notes that
\begin{align*}
\norm{\gamma_{\bb n,j}(\bb u)-(\mathcal P_0\circ\gamma_{\bb n,j})(\bb u)}_{H^{-1/2}(\partial\Gamma_j)} ={}&{} \sup_{0\neq v\in H^{1/2}(\partial\Gamma_j)}\frac{\langle (\operatorname{Id}-\mathcal P_0)(\gamma_{\bb n,j}\bb u), v \rangle_{L^2(\partial\Gamma_j)}}{\norm{v}_{H^{1/2}(\partial \Gamma_j)}}\\
\leq {}&{}\sum_{i=0,\dots,3} \left[\sup_{0\neq v_i\in H^{1/2}(\partial\Gamma_{j,i})}\frac{\langle (\operatorname{Id}-\mathcal P_0)(\gamma_{\bb n,j}\bb u), v_i \rangle_{L^2(\partial\Gamma_{j,i})}}{\norm{v_i}_{H^{1/2}(\partial \Gamma_{j,i})}}\right], 
\end{align*}
where each $\partial \Gamma_{j,i}$ for $i=0,\dots 3$ corresponds to one of the smooth sides of $\Gamma_j.$ In light of Remark \ref{rem::loc::l2partialgammaj}, application of the arguments in the proof of Corollary \ref{H-1/2} yields 
\begin{align*}
\norm{ \gamma_{\bb n,j}(\bb u)-(\mathcal P_0\circ\gamma_{\bb n,j})(\bb u)}_{H^{-1/2}(\partial\Gamma_j)}\lesssim h^s\sum_{i=0,\dots,3}\norm{\gamma_{\bb n,j} (\bb u)}_{H^{s-1/2}(\partial\Gamma_{j,i}) }.
\end{align*}
One can apply the shift property of the normal trace as observed in Remark \ref{rem::localshiftinggamman} and obtains
$$\norm{  \gamma_{\bb n,j}(\bb u)-(\mathcal P_0\circ\gamma_{\bb n,j})(\bb u)}_{H^{-1/2}(\partial\Gamma_j)} \lesssim h^s\sum_{i=0,\dots,3}\norm{\gamma_{\bb n,j} (\bb u)}_{H^{s-1/2}(\partial\Gamma_{j,i}) }\lesssim h^s \norm{\bb u}_{\bb H^s(\div_\Gamma,\Gamma_j)}.$$
This finally} yields
\begin{align*}
\norm{\bb u - \bb w}_{\bb H^0(\div_\Gamma,\Gamma_j)} \lesssim 
h^s \norm{\bb u}_{\bb H^s(\div_\Gamma,\Gamma_j)},
\end{align*}
which along with \eqref{eq:infw} concludes the proof.
}\qed
\end{proof}

{To complete the proof of Theorem \ref{thm::hdiv}, we need to introduce {one last definition}.}

\begin{definition}[Interface Approximation]\label{interfaceapprox}
Given a function {$\bb f\in \bb {H}^1(\Gamma_j)$}, its \emph{interface approximation} is given by $\bb \varphi = \bb \varphi_0+\bb \varphi_1$, where $\bb \varphi_1\in{N^{\S}_j}$ is given as the solution to the problem
        \begin{align}
                    \langle{\gamma_{\bb n,j}(\bb \varphi_1),\gamma_{\bb n,j}(\bb b)\rangle}_{L^2(\partial \Gamma_j)} & = \langle{\gamma_{\bb n,j}(\bb f),\gamma_{\bb n,j}(\bb b)\rangle}_{L^2(\partial\Gamma_j)},\qquad\forall \bb b\in N^{\bb \S}_j,\label{25}
                    \intertext{and $\bb \varphi_0\in K^0_j$ given by}
                    \langle{\bb \varphi_0,\bb v\rangle}_{\bb H^0(\div_\Gamma,\Gamma_j)} &= \langle{\bb f-\bb \varphi_1,\bb v\rangle}_{\bb H^0(\div_\Gamma,\Gamma_j)},\qquad \forall \bb v\in K^0_j.\label{x0}
                \end{align}
\end{definition}
The interface approximation $\bb \varphi$ is chosen as the $\bb H^0(\div_\Gamma,\Gamma_j)$-optimal approximation of $\bb f$ such that the outgoing fluxes $\bb \varphi_1$ consist of the $L^2(\partial \Gamma_j)$-optimal approximation in the discrete sense. Note that, due to the construction of the spline space, the same is obtained if one were to apply this projection to each side of $\partial\Gamma_j$ separately. Since $\bb \varphi_1$ as above is well-defined and the problem in \eqref{x0} is well-posed, it is clear that $\bb \varphi$ is well defined. Using this notion, we can now provide the following result.

\begin{remark}
    We remark that we require regularity of $\bb f$ in Definitions \ref{def::helpingprojection} and \ref{interfaceapprox} only for \eqref{proof::eq::outgoingfluxL2} and \eqref{25} to be well defined in the sense of $L^2$-orthogonality. Both definitions are merely technical tools to provide {an} estimate w.r.t.~the $\bb H^{-1/2}_\times(\div_\Gamma,\Gamma)$-orthogonal projection, which, by density arguments, does not depend on the extra regularity.
\end{remark}

We now have the required tools to show the desired convergence property.

\begin{proof}[\textbf{Proof of Theorem \ref{thm::hdiv}}]
       
        Fix an index $0\leq j< N$, and, for now, assume $\bb f$ to be regular enough for Theorem \ref{lem::multiconv} to be applicable. {Specifically, this means that $\bb f$ is smooth enough for Definitions \ref{def::helpingprojection} and \ref{interfaceapprox} to be well defined.}

        The triangle inequality with the interface approximation $\bb \varphi$ of $\bb f$ on $\Gamma_j$  as intermediate element yields
        \begin{align}
            \begin{aligned}
                \norm{\bb f|_{\Gamma_j}-\pi_j (\bb f|_{\Gamma_j})}_{\bb H^{-1/2}(\div_\Gamma,\Gamma_j)}\leq{}&{} \norm{\bb f|_{\Gamma_j}-\bb \varphi}_{\bb H^{-1/2}(\div_\Gamma,\Gamma_j)} \\ &\qquad+\norm{\bb \varphi-\pi_j \bb f}_{\bb H^{-1/2}(\div_\Gamma,\Gamma_j)}.\label{proof::eq::term1}
            \end{aligned}
        \end{align}
        Let $\mathcal P_0$ denote the $L^2(\partial\Gamma_j)$ orthogonal projection onto {$\gamma_{\bb n,j}(\bb\S^1_{\bb p,\bb \Xi}(\Gamma_j))$.} 
		For the first term, we apply Lemma \ref{LemmaToBeDone} with $\zeta = \gamma_{\bb n,j}( \bb f )- \mathcal P_0 (\gamma_{\bb n,j}(\bb f))$ and $\xi = \bb f|_{\Gamma_j}- \bb \varphi. $ 
		{Note that the required assumptions are satisfied: Galerkin orthogonality yields that $\langle\zeta,1\rangle_{H^{-1/2}(\partial\Gamma_j)\times H^{1/2}(\partial\Gamma_j)}=0$ due to $1\in\gamma_{\bb n,j}\big(\bb\S^1_{\bb p,\bb \Xi}(\Gamma_j)\big)$ and $\xi$ is the solution to the problem given in Lemma \ref{LemmaToBeDone} due to definition of the interface approximation, see \eqref{x0}.}
        This results in
        \begin{align}
            \begin{aligned}
                    \norm{\bb f|_{\Gamma_j}- \bb \varphi }_{\bb H^{-1/2}(\div_\Gamma, \Gamma_j)}
                    &\lesssim  \norm{\gamma_{\bb n,j}( \bb f )- \mathcal P_0 (\gamma_{\bb n,j}(\bb f))}_{H^{-1}(\partial\Gamma_j)},\\
                    \norm{\bb f|_{\Gamma_j}- \bb \varphi }_{\bb H^{0}(\div_\Gamma, \Gamma_j)} &\lesssim  \norm{\gamma_{\bb n,j}( \bb f )- \mathcal P_0 (\gamma_{\bb n,j}(\bb f))}_{H^{-1/2}(\partial\Gamma_j)}. \label{proof::eq::lemmaused}
            \end{aligned}
        \end{align}
        {Then, application of duality arguments as \eqref{classicaldualityhalforderconvergence}, compare also} {\cite[Eq. (103)]{Buffa_2003ab}}, yield the estimate
        \begin{align}
                \norm{\bb f|_{\Gamma_j}- \bb \varphi }_{\bb H^{-1/2}(\div_\Gamma, \Gamma_j)}\lesssim{}&{} \norm{\gamma_{\bb n,j}( \bb f) - \mathcal P_0(\gamma_{\bb n,j}(\bb f))}_{H^{-1}(\partial\Gamma_j)}\notag\\
                \lesssim{}&{}   h^{1/2}\norm{\gamma_{\bb n,j}(\bb f) - \mathcal P_0(\gamma_{\bb n,j}(\bb f))}_{H^{-1/2}(\partial\Gamma_j)}\notag\\ 
                \lesssim{}& {} h^{1/2} \norm{\bb f - \pi_j\bb f}_{{\bb H}^{0}(\div_\Gamma,\Gamma_j)}. 
                \label{proof::eq::term2}
        \end{align}
        Here the last inequality is due to the continuity of the normal trace, and the fact that, by \eqref{proof::eq::outgoingfluxL2}, $\mathcal P_0\circ \gamma_{\bb n,j} = \gamma_{\bb n,j}\circ \pi_j$. Lemma \ref{lem::quasioptimal} yields 
        \begin{align}
               \norm{\bb f|_{\Gamma_j}- \bb \varphi }_{\bb H^{-1/2}(\div_\Gamma, \Gamma_j)}\lesssim h^{s+1/2} \norm{\bb f}_{{\bb H}^{s}(\div_\Gamma,\Gamma_j)}\label{proof::first_term_patchwise}
        \end{align}
        for ${1\leq} s \leq p.$ 

        To estimate the second term  of \eqref{proof::eq::term1} we note that {by Definition~\ref{interfaceapprox} and \eqref{proof::eq::outgoingfluxL2} we know that $\gamma_{\bb n,j}(\bb \varphi) = \gamma_{\bb n,j}(\pi_j \bb \varphi)$ holds, and} thus
        $(\bb \varphi -\pi_j \bb \varphi) \in K_j^0$ follows.
        {For all $\bb b\in K^\S_j\subset K^0_j$ we find, through application of \eqref{proof::eq::kernelorthogonal}, \eqref{x0} and \eqref{proof::eq::kernelorthogonal} again, that
        \begin{align*}
        	\langle \pi_j \bb \varphi,\bb b\rangle_{\bb H^0(\div_\Gamma,\Gamma_j)} &  = \langle \bb \varphi_0 + \bb \varphi_1,\bb b\rangle_{\bb H^0(\div_\Gamma,\Gamma_j)}\\
        	&= \langle \bb f - \bb\varphi_1 + \bb \varphi_1,\bb b\rangle_{\bb H^0(\div_\Gamma,\Gamma_j)}\\
        	&=\langle \pi_j \bb f,\bb b\rangle_{\bb H^0(\div_\Gamma,\Gamma_j)}.
        \end{align*}
        This, together with $\gamma_{\bb n,j}(\bb \varphi - \bb f)= 0$ which we know since \eqref{proof::eq::outgoingfluxL2} and \eqref{25} coincide, implies}
        \begin{align}
            \pi_j\bb \varphi = \pi_j\bb f,\label{proof::eq::thecoincide}
        \end{align} which yields
        {$(\bb\varphi -\pi_j\bb f)\in K^0_j\subset K_j^{-1/2}.$ Thus, it follows that}
        \begin{align*}
            \norm{\bb \varphi-\pi_j\bb f}_{\bb H^{-1/2}(\div_\Gamma,\Gamma_j)}= &{}\sup_{0\neq \bb v \in  (K^{-1/2}_j)'} \frac{\langle\bb \varphi-\pi_j\bb f, \bb v\rangle_{\bb H^0(\div_\Gamma,\Gamma_j)}}{\norm{\bb v}_{(K^{-1/2}_j)'}}
        \end{align*}
        holds.
        We stress that $K_j^0\subseteq K^{-1/2}_j$ and that $K^{-1/2}_j $ is a closed subspace of $\bb{H}^{-1/2}(\div_\Gamma,\Gamma_j).$
        Lemma \ref{lem::isomorph} and the fact that on the kernel of $\gamma_{\bb n,j}$ the projector $\pi_j$ coincides with the $\bb H^0(\div_\Gamma,\Gamma_j)$-orthogonal projection {onto $\bb\S^{1}_{\bb\Xi,\bb p}(\Gamma)$, cf.~\eqref{proof::eq::kernelorthogonal},} allow us to apply the Aubin-Nitsche technique to the above. From this follows
        \begin{align*}
            {}\sup_{0\neq\bb v \in  (K^{-1/2}_j)'} \frac{\langle \bb \varphi-\pi_j\bb f,\bb v\rangle_{\bb H^0(\div_\Gamma,\Gamma_j)}}{\norm{\bb v}_{(K^{-1/2}_j)'}}
            \lesssim{} &{} \sup_{0\neq\bb w\in K^{1/2}_j}\frac{\langle \bb \varphi-\pi_j\bb f,\bb w \rangle_{\bb H^0(\div_\Gamma,\Gamma_j)}}{\norm{\bb w}_{\bb H^{1/2}(\div_\Gamma,\Gamma_j)}}\\
            ={}& {}  \sup_{0\neq\bb w\in K^{1/2}_j}\frac{\langle \bb \varphi-\pi_j\bb f,\bb w - \pi_j \bb w \rangle_{\bb H^0(\div_\Gamma,\Gamma_j)}}{\norm{\bb w}_{\bb H^{1/2}(\div_\Gamma,\Gamma_j)}}\\
            \leq{}&{} \norm{\bb \varphi-\pi_j\bb f}_{\bb H^0(\div_\Gamma,\Gamma_j)}\sup_{0\neq\bb w\in K^{1/2}_j}\frac{\norm{\bb w - \pi_j \bb w }_{\bb H^0(\div_\Gamma,\Gamma_j)}}{\norm{\bb w}_{\bb H^{1/2}(\div_\Gamma,\Gamma_j)}}.
        \end{align*}
        {Note that $\pi_j$ is applicable, since \eqref{proof::eq::outgoingfluxL2} is well defined due to $\gamma_{\bb n,j}\bb w = 0$ almost everywhere. Corollary \ref{cor::Approxcor}} yields
        \begin{align}       
            \norm{\bb \varphi- \pi_j \bb f}_{\bb H^{-1/2}(\div_\Gamma,\Gamma_j)}\lesssim{} &{}  h^{1/2}\norm{\bb \varphi-\pi_j\bb f}_{\bb H^0(\div_\Gamma,\Gamma_j)}  \label{proof::eq::term3}\\
        \lesssim{} &{}  h^{1/2} \norm{\bb \varphi- \bb f|_{\Gamma_j}}_{\bb H^0(\div_\Gamma,\Gamma_j)} +  h^{1/2} \norm{\bb f|_{\Gamma_j}-\pi_j\bb f}_{\bb H^0(\div_\Gamma,\Gamma_j)}.
        \end{align}
        The second term can again be estimated by application of Lemma \ref{lem::quasioptimal}. 
        For the first term, we apply the second equation of \eqref{proof::eq::lemmaused}, which we can estimate in complete analogy to \eqref{proof::eq::term2}, which yields 
        \begin{align}
           \norm{\bb \varphi- \pi_j \bb f}_{\bb H^{-1/2}(\div_\Gamma,\Gamma_j)}\lesssim{} & h^{s+1/2} \norm{\bb f}_{\bb H^s(\div_\Gamma,\Gamma_j)}. \label{proof::second_term_patchwise}
        \end{align}
        Collecting \eqref{proof::first_term_patchwise} and \eqref{proof::second_term_patchwise} and estimating the terms of \eqref{proof::eq::term1} yields the patchwise estimate
        \begin{align*}
            \norm{\bb f- \pi_j \bb f}_{\bb H^{-1/2}(\div_\Gamma,\Gamma_j)}\lesssim{} & h^{s+1/2} \norm{\bb f}_{\bb H^s(\div_\Gamma,\Gamma_j)}.
        \end{align*}
        Since {$\bb H_\mathrm{pw}^{-1/2}(\div_\Gamma,\Gamma) $} is continuously embeddable in $\bb H^{-1/2}_\times(\div_\Gamma,\Gamma)$ we arrive at the corresponding global assertion
        \begin{align}
            \norm{\bb f-\pi \bb f}_{\bb H_\times^{-1/2}(\div_\Gamma,\Gamma)}\lesssim h^{s+1/2} \norm{\bb f}_{{\bb H}_{{\mathrm{pw}}}^s(\div_\Gamma,\Gamma)}, \qquad {1{}\leq{} } s\leq p,
        \end{align}
        by properties of $\pi$.
        Now stability of $\mathcal P_\times$ w.r.t. $\bb H_\times^{-1/2}(\div_\Gamma,\Gamma)$ and Céa's Lemma yields the estimates 
        \begin{align}
            \norm{\bb f- \mathcal P_\times\bb f}_{\bb H_\times^{-1/2}(\div_\Gamma,\Gamma)} & \lesssim \norm{\bb f}_{{\bb H}_\times^{-1/2}(\div_\Gamma,\Gamma)},\\ 
            \norm{\bb f- \mathcal P_\times\bb f}_{\bb H_\times^{-1/2}(\div_\Gamma,\Gamma)} & \lesssim h^{1/2} \norm{\bb f}_{{\bb H}^0(\div_\Gamma,\Gamma)},
            \intertext{as well as}
            \norm{\bb f- \mathcal P_\times\bb f}_{\bb H_\times^{-1/2}(\div_\Gamma,\Gamma)} & \lesssim h^{1/2}\norm{\bb f}_{  {\bb H}^{0}(\div_\Gamma,\Gamma)},\\ 
            \norm{\bb f- \mathcal P_\times\bb f}_{\bb H_\times^{-1/2}(\div_\Gamma,\Gamma)} & \lesssim h^{s+1/2} \norm{\bb f}_{ {\bb H}^s_{\text{pw}}(\div_\Gamma,\Gamma)}.
        \end{align}
        By density of regular functions in $\bb H_\times^{-1/2}(\div_\Gamma,\Gamma)$ and continuity of the orthogonal projection, the results carry over to non-smooth $\bb f$. Now we can use interpolation to generalise the result to all $-1/2\leq s \leq p$. This can be done thanks to Appendix 2 of \cite{Buffa_2003ab}, which proves that $\bb H_\times^{-1/2}(\div_\Gamma,\Gamma)$ and $\bb H^0(\div_\Gamma,\Gamma)$ induce an interpolation scale, i.e.,~can be handled similarly to Lemma \ref{interpolationlemma}. {Specifically, see \cite[Thm.~4.12]{Buffa_2003ab} where the notation translates to ours via 
        $X=\bb H^{-1/2}_\times(\div_\Gamma,\Gamma)$ and $X^s=\bb H^{s}_\times(\div_\Gamma,\Gamma)$.}
        \qed
    \end{proof}

\FloatBarrier
\section{Conclusion}\label{sec::conclusion}

We have derived multipatch approximation results of the spline complex w.r.t.~the norms required by boundary- and finite element methods.

Let the functions $f_0$, $\bb f_1$, $f_2$ be regular enough for the norms on both left and right-hand side of the following estimates to be well defined, see also Lemma \ref{lem::multipatch::theycommute::and::regularity}.
For multipatch boundaries $\Gamma$ in accordance with Assumptions \ref{ass::knotvecs} and \ref{ass::geometry}, we proved
\begin{align}
        \inf_{f_h\in \S^0_{\bb p,\bb \Xi}(\Gamma)}\norm{ f_0 - f_h}_{H^{1/2}(\Gamma)}&\lesssim h^{s-1/2}\norm{f_0}_{{H}^{s}_{{\mathrm{pw}}}(\Gamma)}& {2}\leq {}&{}s \leq p+1,\label{2d:1}\\
        \inf_{\bb f_h\in \bb \S^1_{\bb p,\bb \Xi}(\Gamma)}\norm{ \bb f_1 - \bb f_h}_{\bb H_\times^{-1/2}(\div_\Gamma,\Gamma)}&\lesssim  h^{s+1/2}\norm{\bb f_1}_{{\bb H}_\times^s(\div_\Gamma,\Gamma)}&-1/2\leq {}&{}s \leq 0,\label{2d:2}\\
        \inf_{\bb f_h\in \bb \S^1_{\bb p,\bb \Xi}(\Gamma)}\norm{ \bb f_1 - \bb f_h}_{\bb H_\times^{-1/2}(\div_\Gamma,\Gamma)}&\lesssim  h^{s+1/2}\norm{\bb f_1}_{  {\bb H}_{{\mathrm{pw}}}^s(\div_\Gamma,\Gamma)}&0\leq {}&{}s \leq p,\label{2d:22}\\
        \inf_{f_h\in \S^2_{\bb p,\bb \Xi}(\Gamma)}\norm{ f_2 - f_h}_{H^{-1/2}(\Gamma)}&\lesssim h^{s+1/2}\norm{f_2}_{{H}^{s}_{{\mathrm{pw}}}(\Gamma)}&{0}\leq {}&{}s \leq p.\label{2d:3}
\end{align}
Here, \eqref{2d:1} follows from Corollary \ref{H1/2}, \eqref{2d:2} and \eqref{2d:22} follow from Theorem \ref{thm::hdiv}, \eqref{2d:3} follows from Corollary \ref{H-1/2}. 
Moreover, we can apply these results for finite element methods as well. By extension of the tensor product structure in the construction of spline spaces and interpolation operators by one dimension, see \appendA{}, we find for multipatch domains $\Omega\subseteq \mathbb R^d$, with $d=2,3$, the estimates
\begin{align*}
   \inf_{f_h\in \S^0_{\bb p,\bb \Xi}(\Omega)}\norm{ f_3 - f_h}_{H^{1}(\Omega)}&\lesssim h^{s-1}\norm{f_3}_{{H}^{s}_{{\mathrm{pw}}}(\Omega)}&{d}\leq {}&{}s \leq p+1,\\
    \inf_{\bb f_h\in \bb \S^1_{\bb p,\bb \Xi}(\Omega)}\norm{ \bb f_4 - \bb f_h}_{\bb H^0(\div,\Omega)}&\lesssim  h^{s}\norm{\bb f_4}_{{\bb H}^s_{{\mathrm{pw}}}(\div,\Omega_j)}&{1}\leq {}&{}s \leq p,\\
    \inf_{f_h\in \S^2_{\bb p,\bb \Xi}(\Omega)}\norm{ f_5 - f_h}_{L^2(\Omega)}&\lesssim h^{s}\norm{f_5}_{{H}^{s}_{{\mathrm{pw}}}(\Omega)}&0\leq {}&{}s \leq p,
\end{align*}
for $f_3$, $\bb f_4$ and $f_5$ smooth enough for the norms to be defined, as explained in Corollary \ref{cor::femanalogy}. 
Estimates for three-dimensional spaces, including $\bb H(\bb\curl,\Omega),$ follow analogously, cf. Corollary~\ref{volumetricmulticonv} in Appendix A.
We can drop the regularity requirements from Theorem \ref{lem::multiconv}, since they are only required by the constructed quasi-interpolants, and not by the orthogonal projection w.r.t.~the corresponding Sobolev spaces, see Section \ref{sec::tracespaces}.

Taking into account the three-dimensional generalisation of the construction in Section \ref{subsec::approx}, see \appendA{}, we now have access to a discretisation of the diagram in Figure \ref{fig::classicaldeRham}, given by
\begin{equation}
   \begin{tikzcd}[row sep = 2em,column sep = 1.3cm]
            \S^0_{\tilde{\bb p},\bb {\tilde \Xi}}(\Omega)\ar{d}[description]{\gamma_0}\ar{r}[description]{\bb\grad}&
            \bb \S^1_{\tilde{\bb p},\bb {\tilde \Xi}}(\Omega)\ar{r}[description]{\bb\curl}\ar{d}[description]{\bb{\gamma}_t}&
            \bb \S^2_{\tilde{\bb p},\bb {\tilde \Xi}}(\Omega)\ar{d}[description]{\gamma_{\bb n}}\ar{r}[description]{\div} & \S^3_{\tilde{\bb p},\bb {\tilde \Xi}}(\Omega)\\
            \S^0_{{\bb p},\bb \Xi}(\Gamma)\ar{r}[description]{{\vcurl_\Gamma}} & \bb \S^1_{\bb p,\bb \Xi}(\Gamma)\ar{r}[description]{\div_\Gamma} & \S^2_{\bb p,\bb \Xi}(\Gamma)
    \end{tikzcd}\label{Fig::discreteDeRham}
\end{equation}
for suitable choices of (lists of tuples of) polynomial degrees $\tilde{\bb p},\bb p$ and knot vectors $\tilde{\bb \Xi}, \bb \Xi$, and under the assumption that $\Omega$ is given as a multipatch domain. {Note that a corresponding discretisation of $\bb H^{-1/2}(\curl_\Gamma,\Gamma)$ can be obtained in complete analogy to the construction of $\bb\S^1_{\bb p,\bb \Xi}(\Gamma)$.}

To this end, we know that for \emph{any} problem formulated within the isogeometric framework that enjoys a discrete inf-sup condition or a variant of Céa's Lemma w.r.t.~the norms above, we can expect a convergence of optimal order w.r.t.~$h$-refinement \cite{Xu_2002aa}.
Note, however, that the orthogonal projection will, in general, not have the commuting diagram property in the sense of Lemma \ref{lem::multipatch::theycommute::and::regularity}. This distinction is critical for existence and uniqueness proofs for problems requiring conforming discretisations.

\bibliography{local}

\begin{thebibliography}{10}
\providecommand{\url}[1]{{#1}}
\providecommand{\urlprefix}{URL }
\expandafter\ifx\csname urlstyle\endcsname\relax
  \providecommand{\doi}[1]{DOI~\discretionary{}{}{}#1}\else
  \providecommand{\doi}{DOI~\discretionary{}{}{}\begingroup
  \urlstyle{rm}\Url}\fi

\bibitem{Adams_1978aa}
Adams, R.A.: Sobolev spaces.
\newblock Pure and Applied Mathematics. Academic Press, New York (1978)

\bibitem{Beer_2017aa}
Beer, G., Mallardo, V., Ruocco, E., Marussig, B., Zechner, J., D\"unser, C.,
  Fries, T.P.: Isogeometric boundary element analysis with elasto-plastic
  inclusions. part 2: 3-d problems.
\newblock Computer Methods in Applied Mechanics and Engineering
  \textbf{315}(Supplement C), 418--433 (2017)

\bibitem{Veiga_2014aa}
{Beir\~ao da Veiga}, L., Buffa, A., Sangalli, G., V{\'a}zquez, R.:
  {Mathematical analysis of variational isogeometric methods}.
\newblock Acta Numerica \textbf{23}, 157--287 (2014)

\bibitem{Bergh_1976aa}
Bergh, J., Löfström, J.: Interpolation spaces: An introduction.
\newblock Grund\-lehren der mathematischen Wissenschaften. Springer,
  Berlin-Heidelberg (1976)

\bibitem{Boffi2013}
Boffi, D., Brezzi, F., Fortin, M.: Mixed Finite Element Methods and
  Applications.
\newblock Springer, Berlin-Heidelberg (2013)

\bibitem{Bontinck_2017ag}
Bontinck, Z., Corno, J., {De Gersem}, H., Kurz, S., Pels, A., Sch\"ops, S.,
  Wolf, F., de~Falco, C., D\"olz, J., V\'azquez, R., R\"omer, U.: Recent
  advances of isogeometric analysis in computational electromagnetics.
\newblock ICS Newsletter (International Compumag Society) \textbf{24}(3)
  (2017).
\newblock
  \urlprefix\url{http://www.compumag.org/jsite/images/stories/newsletter}.
\newblock Available as preprint via arXiv, e-print 1709.06004

\bibitem{Bossavit_1988aa}
Bossavit, A.: Whitney forms: a class of finite elements for three-dimensional
  computations in electromagnetism.
\newblock IEEE Proceedings \textbf{135}(8), 493--500 (1988)

\bibitem{Bossavit_1998aa}
Bossavit, A.: Computational Electromagnetism: Variational Formulations,
  Complementarity, Edge Elements.
\newblock Academic Press, San Diego (1998)

\bibitem{BrennerScott}
Brenner, S., Scott, L.: The Mathematical Theory of Finite Element Methods.
\newblock Springer Berlin-Heidelberg (2008)

\bibitem{Buffa_2003ab}
Buffa, A., Christiansen, S.: The electric field integral equation on
  {L}ipschitz screens: definitions and numerical approximation.
\newblock Numerische Mathematik \textbf{94}(2), 229--267 (2003)

\bibitem{Buffa_2001aa}
Buffa, A., Ciarlet, P.: On traces for functional spaces related to {M}axwell's
  equations part {I}: An integration by parts formula in {L}ipschitz polyhedra.
\newblock Mathematical Methods in the Applied Sciences \textbf{24}(1), 9--30
  (2001)

\bibitem{Buffa_2001ac}
Buffa, A., Ciarlet, P.: On traces for functional spaces related to {M}axwell's
  equations part {II}: {H}odge decompositions on the boundary of {L}ipschitz
  polyhedra and applications.
\newblock Mathematical Methods in the Applied Sciences \textbf{24}(1), 31--48
  (2001)

\bibitem{Buffa_2002aa}
Buffa, A., Costabel, M., Sheen, D.: On traces for {H(curl,$\Omega$)} in
  {Lipschitz} domains.
\newblock Journal of Mathematical Analysis and Applications \textbf{276}(2),
  845--867 (2002)

\bibitem{Buffa_2003aa}
Buffa, A., Hiptmair, R.: {Galerkin} boundary element methods for
  electromagnetic scattering.
\newblock Topics in Computational Wave Propagation pp. 83--124 (2003)

\bibitem{Buffa_2011aa}
Buffa, A., Rivas, J., Sangalli, G., Vázquez, R.: Isogeometric discrete
  differential forms in three dimensions.
\newblock SIAM Journal on Numerical Analysis \textbf{49}(2), 818--844 (2011)

\bibitem{Buffa_2010aa}
Buffa, A., Sangalli, G., V\'{a}zquez, R.: Isogeometric analysis in
  electromagnetics: B-splines approximation.
\newblock Computer Methods in Applied Mechanics and Engineering \textbf{199},
  1143--1152 (2010)

\bibitem{Buffa_2015aa}
Buffa, A., Vázquez, R.H., Sangalli, G., Beirão~da Veiga, L.: Approximation
  estimates for isogeometric spaces in multipatch geometries.
\newblock Numerical Methods for Partial Differential Equations \textbf{31}(2),
  422--438 (2015)

\bibitem{Ciarlet_2002aa}
Ciarlet, P.G.: The Finite Element Method for Elliptic Problems, \emph{Classics
  in applied mathematics}, vol.~40, 2nd edn.
\newblock Society for Industrial Mathematics (2002)

\bibitem{Costabel_2000aa}
Costabel, M., Dauge, M.: Singularities of electromagnetic fields in polyhedral
  domains.
\newblock Archive for Rational Mechanics and Analysis \textbf{151}(3), 221--276
  (2000)

\bibitem{Costabel_polynomialextension}
Costabel, M., Dauge, M., Demkowicz, L.: Polynomial extension operators for
  ${H}^1$, ${H}(\operatorname{curl})$ and ${H}(\operatorname{div})$-spaces on a
  cube.
\newblock Math. Comp pp. 1967--1999 (2008)

\bibitem{Cottrell_2009aa}
Cottrell, J.A., Hughes, T.J.R., Bazilevs, Y.: Isogeometric Analysis: Toward
  Integration of {CAD} and {FEA}.
\newblock Wiley, West Sussex (2009)

\bibitem{Dhaeseleer_1991aa}
{D'H}aeseleer, W., Hitchon, W., Callen, J., Shohet, J.: Flux coordinates and
  magnetic field structure: a guide to a fundamental tool of plasma structure.
\newblock Springer series in computational physics. Springer, Berlin-Heidelberg
  (1991)

\bibitem{Nezza_2012aa}
{Di Nezza}, E., Palatucci, G., Valdinoci, E.: Hitchhiker's guide to the
  fractional {S}obolev spaces.
\newblock Bulletin des Sciences Mathématiques \textbf{136}(5), 521--573 (2012)

\bibitem{Doelz_2017aa}
Dölz, J., Harbrecht, H., Kurz, S., Schöps, S., Wolf, F.: A fast isogeometric
  {BEM} for the three dimensional {Laplace}- and {Helmholtz} problems.
\newblock Computer Methods in Applied Mechanics and Engineering
  \textbf{330}(Supplement C), 83 -- 101 (2018)

\bibitem{Doelz_2016aa}
Dölz, J., Harbrecht, H., Peters, M.: An interpolation-based fast multipole
  method for higher order boundary elements on parametric surfaces.
\newblock International Journal for Numerical Methods in Engineering
  \textbf{108}(13) (2016)

\bibitem{Gerritsma2010}
Gerritsma, M.: Edge functions for spectral element methods.
\newblock In: Lecture Notes in Computational Science and Engineering, pp.
  199--207. Springer Berlin Heidelberg (2010)

\bibitem{Girault1986}
Girault, V., Raviart, P.A.: Finite element methods for {N}avier–{S}tokes
  equations  (1986)

\bibitem{Hackbusch_2002aa}
Hackbusch, W., Börm, S.: {$\mathcal{H}^2$}-matrix approximation of integral
  operators by interpolation.
\newblock Applied Numerical Mathematics \textbf{43}(1), 129--143 (2002)

\bibitem{Harbrecht._2001aa}
Harbrecht., H.: Wavelet {Galerkin} schemes for the boundary element method in
  three dimensions.
\newblock Ph.D. thesis, Technische Universit\"at Chemnitz (2001)

\bibitem{Harbrecht_2013aa}
Harbrecht, H., Peters, M.: Comparison of fast boundary element methods on
  parametric surfaces.
\newblock Computer Methods in Applied Mechanics and Engineering \textbf{261},
  39--55 (2013)

\bibitem{Harbrecht_2010aa}
Harbrecht, H., Randrianarivony, M.: From computer aided design to wavelet
  {BEM}.
\newblock Computing and Visualization in Science \textbf{13}, 69--82 (2010)

\bibitem{Hiptmair_2002aa}
Hiptmair, R.: Finite elements in computational electromagnetism.
\newblock Acta Numerica \textbf{11}, 237--339 (2002)

\bibitem{Hughes_2005aa}
Hughes, T.J.R., Cottrell, J.A., Bazilevs, Y.: Isogeometric analysis: {CAD},
  finite elements, {NURBS}, exact geometry and mesh refinement.
\newblock Computer Methods in Applied Mechanics and Engineering \textbf{194},
  4135--4195 (2005)

\bibitem{Kurz_2007aa}
Kurz, S., Rain, O., Rjasanow, S.: Fast boundary element methods in
  computational electromagnetism.
\newblock In: M.~Schanz, O.~Steinbach (eds.) Boundary Element Analysis:
  Mathematical Aspects and Applications, pp. 249--279. Springer,
  Berlin-Heidelberg (2007)

\bibitem{Lee_1996aa}
Lee, E.T.Y.: Marsden's identity.
\newblock Computer Aided Geometric Design \textbf{13}(4), 287--305 (1996)

\bibitem{Marussig_2015aa}
Marussig, B., Zechner, J., Beer, G., Fries, T.P.: Fast isogeometric boundary
  element method based on independent field approximation.
\newblock Computer Methods in Applied Mechanics and Engineering
  \textbf{284}(0), 458--488 (2015)

\bibitem{McLean_2000aa}
McLean, W.: Strongly elliptic systems and boundary integral equations.
\newblock Cambridge University Press, Cambridge, United Kingdom (2000)

\bibitem{Monk_1993aa}
Monk, P.: An analysis of {N}édélec's method for the spatial discretization of
  {M}axwell's equations.
\newblock Journal of Computational and Applied Mathematics \textbf{47}(1),
  101--121 (1993)

\bibitem{Monk_2003aa}
Monk, P.: Finite Element Methods for {M}axwell's Equations.
\newblock Oxford University Press, Oxford (2003)

\bibitem{Nedelec_1980aa}
Nédélec, J.C.: Mixed finite elements in ${R}^3$.
\newblock Numerische Mathematik \textbf{35}(3), 315--341 (1980)

\bibitem{Peterson_2006aa}
Peterson, A.F.: Mapped vector basis functions for electromagnetic integral
  equations.
\newblock Synthesis Lectures on Computational Electromagnetics \textbf{1}(1),
  1--124 (2006)

\bibitem{Peterson_1995aa}
Peterson, A.F., Aberegg, K.R.: Parametric mapping of vector basis functions for
  surface integral equation formulations.
\newblock Appl. Comput. Electromagn. Soc. J. \textbf{10}, 107 -- 115 (1995)

\bibitem{Piegl_1997aa}
Piegl, L., Tiller, W.: The {NURBS} Book, 2nd edition edn.
\newblock Springer, Berlin-Heidelberg (1997)

\bibitem{Quarteroni_1999aa}
Quarteroni, A., Valli, A., Valli, P.: Domain Decomposition Methods for Partial
  Differential Equations.
\newblock Numerical Mathematics and Scie. Clarendon Press (1999)

\bibitem{Sauter_2010aa}
Sauter, S., Schwab, C.: Boundary Element Methods.
\newblock Springer Series in Computational Mathematics. Springer,
  Berlin-Heidelberg (2010)

\bibitem{Schumaker_2007aa}
Schumaker, L.L.: Spline functions: Basic theory.
\newblock Cambridge Mathematical Library. Cambridge University Press,
  Cambridge, United Kingdom (2007)

\bibitem{Simpson_2012aa}
Simpson, R.N., Bordas, S.P.A., Trevelyan, J., Rabczuk, T.: A two-dimensional
  isogeometric boundary element method for elastostatic analysis.
\newblock Computer Methods in Applied Mechanics and Engineering
  \textbf{209–212}, 87--100 (2012)

\bibitem{Simpson_2017aa}
{Simpson}, R.N., {Liu}, Z., {V\'azquez}, R., {Evans}, J.A.: {An isogeometric
  boundary element method for electromagnetic scattering with compatible
  B-spline discretizations}.
\newblock arXiv e-prints  (2017).
\newblock 1704.07128

\bibitem{Steinbach_2008aa}
Steinbach, O.: Numerical Approximation Methods for Elliptic Boundary Value
  Problems.
\newblock Finite and Boundary Elements. Springer, New York (2008)

\bibitem{Weggler_2011aa}
Weggler, L.: High order boundary element methods.
\newblock Ph.D. thesis, Universit\"at des Saarlandes, Saarbr\"ucken (2011)

\bibitem{wloka_1987aa}
Wloka, J.: Partial Differential Equations.
\newblock Cambridge University Press (1987)

\bibitem{Xu_2002aa}
Xu, J., Zikatanov, L.: Some observations on {Babu\v{s}ka} and {Brezzi}
  theories.
\newblock Numerische Mathematik \textbf{94}(1), 195--202 (2002)

\bibitem{Zaglmayr_2006aa}
Zaglmayr, S.: High order finite element methods for electromagnetic field
  computation.
\newblock Ph.D. thesis, Linz, Austria (2006)

\end{thebibliography}
\quad\\[-1cm]
\section*{Appendix A}
    \label{subsec::threeD}
\nocite{Costabel_2000aa}
    All the presented estimates are applicable to achieve three-dimensional estimates as well, going back to \cite{Veiga_2014aa,Buffa_2011aa}. We will briefly go over the construction and state the result corresponding to Theorem \ref{lem::multiconv}.

    For $p>0$ we define the spline complex on $[0,1]^3$ via
    \begin{align}
        \begin{aligned}
            \S^0_{\bb p,\bb \Xi}([0,1]^3)\coloneqq {}&{} S_{p_1,p_2,p_3}(\Xi_1,\Xi_2,\Xi_3),\\
            \bb \S^1_{\bb p,\bb \Xi}([0,1]^3)\coloneqq {}&{} S_{p_1-1,p_2,p_3}(\Xi_1',\Xi_2,\Xi_3)
             \times \\ &\qquad 
            \times S_{p_1,p_2-1,p_3}(\Xi_1,\Xi_2',\Xi_3) 
            \times \\&\qquad\qquad 
            \times S_{p_1,p_2,p_3-1}(\Xi_1,\Xi_2,\Xi_3'), \\
            \bb \S^2_{\bb {p},\bb \Xi}([0,1]^3)\coloneqq {}&{} S_{p_1,p_2-1,p_3-1}(\Xi_1,\Xi_2',\Xi_3')
             \times \\ &\qquad 
            \times S_{p_1-1,p_2,p_3-1}(\Xi_1',\Xi_2,\Xi_3')
             \times \\&\qquad\qquad 
            \times S_{p_1-1,p_2-1,p_3}(\Xi_1',\Xi_2',\Xi_3),  \\
            \S^3_{\bb p,\bb \Xi}([0,1]^3)\coloneqq {}&{} S_{p_1-1,p_2-1,p_3-1}(\Xi_1',\Xi_2',\Xi_3').
        \end{aligned}
    \end{align}
    Let $f_0,$ $\bb f_1,$ $\bb f_2,$ $f_3$ be sufficiently smooth. We can use the transformations 
    \begin{align}
        \begin{aligned}
        \iota_0(\bb F)(f_0)&\coloneqq f_0\circ\bb F,&
        \iota_1(\bb F)(\bb f_1)&\coloneqq (d\bb F)^\top (\bb f_1\circ \bb F),\\
        \iota_2(\bb F)(\bb f_2)&\coloneqq \det(d\bb F) (d\bb F)^{-1} (\bb f_2\circ \bb F),&
        \iota_3(\bb F)(f_3)&\coloneqq \det(d\bb F) (f_3\circ \bb F),
        \end{aligned}
    \end{align}
    to define the corresponding spaces in the single patch physical domain as in \eqref{def::singlepatchphysical}, cf.~\cite{Hiptmair_2002aa}. Now, the projections  $\tilde \Pi_{\bb p,\bb \Xi,\Omega}^0$, 
$\tilde {\bb\Pi}_{\bb p,\bb \Xi,\Omega}^{1}$, 
$\tilde {\bb\Pi}_{\bb p,\bb \Xi,\Omega}^{2}$, and
$\tilde \Pi_{\bb p,\bb \Xi,\Omega}^3$ w.r.t.~the reference domain for $\bb \Xi = [\Xi_1,\Xi_2,\Xi_3]$  defined in complete analogy to \eqref{def::commtilde}, commute with the differential operators {$\bb\grad,\bb\curl$ and $\div$.}
    By properties of the pullbacks, cf.~\cite[{Sec.~5.1}]{Veiga_2014aa}, this holds for the physical domain as well.
    The three-dimensional global B-spline projections are then defined as
            \begin{align*}
                \tilde\Pi^0_\Omega\coloneqq \hspace{-.2cm}{}&{}\bigoplus_{0\leq j< N}\hspace{-.1cm}  \left((\iota_{0}(\bb F_j))^{-1}\circ\tilde\Pi_{\bb p,\bb \Xi,\Omega}^0\circ\iota_{0}(\bb F_j)\right),&
                \tilde{\bb\Pi}^1_\Omega\coloneqq \hspace{-.2cm}{}&{}\bigoplus_{0\leq j< N}\hspace{-.1cm}  \left((\iota_{1}(\bb F_j))^{-1}\circ\tilde{\bb\Pi}_{\bb p,\bb \Xi,\Omega}^1\circ\iota_{1}(\bb F_j)\right),\\
                \tilde{\bb\Pi}^2_\Omega\coloneqq \hspace{-.2cm}{}&{}\bigoplus_{0\leq j< N}\hspace{-.1cm}  \left((\iota_{2}(\bb F_j))^{-1}\circ\tilde{\bb\Pi}_{\bb p,\bb \Xi,\Omega}^2\circ\iota_{2}(\bb F_j)\right),&
                \tilde\Pi^3_\Omega\coloneqq \hspace{-.2cm}{}&{}\bigoplus_{0\leq j< N}\hspace{-.1cm}  \left((\iota_{3}(\bb F_j)^{-1}\circ\tilde\Pi_{\bb p,\bb \Xi,\Omega}^3\circ\iota_{3}(\bb F_j)\right).
            \end{align*}

    {In complete analogy to} the proof of Theorem \ref{lem::multiconv}, one can achieve the following result for the three-dimensional multipatch spline complex.
    \begin{corollary}\label{volumetricmulticonv}
    Let the volumetric analogue of Assumptions \ref{ass::knotvecs} and \ref{ass::geometry} be satisfied. Assume the functions $f_1,$ $\bb f_2,$ $\bb f_3,$ $f_4$ to be sufficiently smooth, i.e., such that the norms and interpolation operators below are well defined.
    Then one finds, for integers $s$ as below,
    \begin{align*}
        \norm{f_1 - \tilde\Pi^0_\Omega f_1}_{H^r(\Omega)} &\lesssim h^{s-r} \norm{f_1}_{{H}^s_{\mathrm{pw}}(\Omega)},&& 3\leq s\leq p+1,\\
        \norm{\bb f_2 - \tilde{\bb\Pi}^1_\Omega \bb f_2}_{\bb H(\curl,\Omega)} &\lesssim h^s \norm{\bb f_2}_{{\bb H}^s_{\mathrm{pw}}(\curl,\Omega)},&& 2 < s\leq p,\\
        \norm{\bb f_3 - \tilde{\bb\Pi}^2_\Omega \bb f_3}_{\bb H(\div,\Omega)} &\lesssim h^s \norm{\bb f_3}_{{\bb H}^s_{\mathrm{pw}}(\div,\Omega)},&& 1 < s\leq p,\\
        \norm{f_4 - \tilde\Pi^3_\Omega f_4}_{L^2(\Omega)}            &\lesssim h^s \norm{f_4}_{{H}^s_{\mathrm{pw}}(\Omega)},&& 0 \leq s\leq p,
    \end{align*}
    for $r = 0,1$.
    \end{corollary}
\newpage
\bibliographystyle{spmpsci}

\end{document}